\documentclass[11pt]{article}%
\usepackage{amsmath}
\usepackage{amsfonts}
\usepackage{mathrsfs}
\usepackage{amssymb, color}
\usepackage[linkcolor=black,anchorcolor=black,citecolor=black]{hyperref}
\usepackage{graphicx}
\numberwithin{equation}{section}
\usepackage[body={15.5cm,21cm}, top=3cm]{geometry}%
\providecommand{\U}[1]{\protect\rule{.1in}{.1in}}
\providecommand{\U}[1]{\protect \rule{.1in}{.1in}}
\newtheorem{theorem}{Theorem}[section]

\newtheorem{example}[theorem]{Example}

\newtheorem{lemma}[theorem]{Lemma}

\newtheorem{proposition}[theorem]{Proposition}
\newtheorem{remark}[theorem]{Remark}

\newtheorem{assumption}[theorem]{Assumption}
\newenvironment{proof}[1][Proof]{\noindent \textbf{#1.} }{\  \rule{0.5em}{0.5em}}
\DeclareMathOperator*{\esssup}{ess\,sup}
\DeclareMathOperator*{\essinf}{ess\,inf}
\def \P{\mathsf{P}}
\def \E{\mathsf{E}}

\newcommand{\fr}[1]{{\textcolor{blue}{#1}}}

\begin{document}
	\title{Optimal Consumption for Recursive Preferences with Local Substitution under Risk}
	\author{ Hanwu Li\thanks{Research Center for Mathematics and Interdisciplinary Sciences, Shandong University,
			lihanwu@sdu.edu.cn, and Frontiers Science Center for Nonlinear Expectations (Ministry of Education), Shandong University. Li's research was supported by the Natural Science Foundation of Shandong Province (No. ZR2022QA022), the Natural Science Foundation of Shandong Province for Excellent Young Scientists Fund Program (Overseas) (No. 2023HWYQ-049) and the National Natural Science Foundation of China (No. 12301178).} \and 
			Frank Riedel\thanks{Center for Mathematical Economics, Bielefeld University, frank.riedel@uni-bielefeld.de, and School of Economics, University of Johannesburg. Riedel’s research was supported by the  German Science Foundation (DFG) – SFB 1283/2 2021 – 317210226.}}
			\date{}
		\maketitle

\begin{abstract}
	We explore intertemporal preferences that are recursive and account for local intertemporal substitution. First, we establish a rigorous foundation for these preferences and analyze their properties. Next, we examine the associated optimal consumption problem, proving the existence and uniqueness of the optimal consumption plan. We present an infinite-dimensional version of the Kuhn-Tucker theorem, which provides the necessary and sufficient conditions for optimality. Additionally, we investigate  quantitative properties and the construction of the optimal consumption plan. Finally, we offer a detailed description of the structure of optimal consumption within a geometric Poisson framework.
\end{abstract}


    \textbf{MSC-classification}: 60H10, 60H30
		
\section{Introduction}	
In this paper, we study recursive intertemporal preferences that allow for intertemporal substitution in a stochastic jump-diffusion framework.
We provide a rigorous foundation for the representing utility functional through backward stochastic differential equations and we establish conditions that ensure existence and uniqueness of the related optimal consumption choice problem. We characterize optimal solutions through first-order conditions and we show how to construct optimal plans based on the concept of a minimal level of satisfaction.

	Intertemporal utility functions are the basic building block for dynamic economic models. One of the most frequently used model is the time-additive expected utility functional. More precisely, consider a cumulative consumption plan $\{C_t\}_{t\in[0,T]}$ given in terms of the consumption rate $\{c_t\}_{t\in[0,T]}$, i.e., $C_t=\int_0^t c_sds$, the time-additive utility functional takes the following form
	\begin{displaymath}
		U(c)=\E\left[\int_0^T u(t,c_t)dt\right].
	\end{displaymath}
 
	Hindy, Huang and Kreps  \cite{HHK} point out severe   shortcomings of such an approach as it violates the natural substitutability of consumption at nearby times. They   introduce intertemporal utility functions of the following form:
	\begin{displaymath}
		U(C)=\E\left[\int_0^T u(t,Y^C_t)dt\right],
	\end{displaymath}
	where $u$ is a continuous felicity function which is increasing and concave in its second argument while $Y^C$ is the investor's level of satisfaction satisfying
	\begin{displaymath}
		Y^C_t=\eta_t+\int_0^t \theta_{t,s}dC_s.
	\end{displaymath}
	Here, $\theta_{t,s}$ can be interpreted as the weight assigned at time $t$ to the consumption made up to time $t$ and $\eta_t$ can be regarded as an exogenously given level of satisfaction. 
 
 Duffie and Epstein \cite{DE} point out that intertemporal utility should be recursive, i.e., take future utility into account.  They introduce the following recursive utility function:
	\begin{displaymath}
		U_t(c)=\E_t\left[\int_t^T f(s,c_s,U_s(c))ds\right],
	\end{displaymath}
	where $f$ is called the aggregator function which aggregates current consumption $c_t$ and future utility to the current utility level. In fact, this kind of recursive utility corresponds to the solution of a backward stochastic differential equation (BSDE) first proposed by Pardoux and Peng \cite{PP90}.
 
	In this paper, we combine the stochastic differential utility of \cite{DE} with the local substitution approach of \cite{HHK}.	
We  focus on  utility functions which can be represented as 
	\begin{displaymath}
		U_t^C=\E_t\left[\int_t^T f(s,Y_s^C,U_s^C)ds\right].
	\end{displaymath}
	The advantage lies in the following aspects. First, the recursiveness describes the future utility's influence on the current marginal utility, i.e., the future behavior has a feedback effect on the current state, which improves the HHK model. Second, this kind of utility is not confined to absolutely continuous consumption plans and it presents the desirable intertemporal substitution properties. Hence, it is more reasonable and suitable to the practical cases compared with the utility functions proposed by Duffie and Epstein \cite{DE} and Hindy, Huang and Kreps \cite{HHK}. Furthermore, different from \cite{DE}, the underlying driven process is not restricted to Brownian motion. In fact, in this paper, we consider a setting consists of a Brownian motion and a Poisson random measure. The recursive utility induced by consumption rate in the presence of L\'{e}vy jumps has been considered by \cite{KS,Ma,S}. In this paper, we first establish the theory of recursive utility induced by the level of satisfaction under a L\'{e}vy setting with the help of BSDE with jumps (see \cite{BBP, Delong,Royer,TL} and the reference therein). To the best of our knowledge, this work is the first attempt to solve such kind of problem.
	
	The objective of the paper is try to maximize the recursive utility over all budget feasible set. First, we establish the existence and uniqueness result for the stochastic optimization problem. The uniqueness relies on the strict concavity of recursive utility. The difficulty for proving the existence is due to the fact that the budget feasible set is not compact. Inspired by the method used in \cite{BR}, we  apply the infinite-dimensional extension of Koml\'{o}s' theorem (see \cite{K99,K67}) to construct the optimal consumption plan. Then, we investigate the characterization of the optimal consumption plans, i.e., the infinite-dimensional version of the Kuhn-Tucker theorem, called the first-order conditions, which extends beyond the Hamilton-Jacobi-Bellman equations framework, see \cite{HH}. It generalizes the first-order conditions in \cite{BR00} and \cite{BR} to a stochastic and recursive setting by applying the stochastic Gronwall inequality.
	
	It is important to note that the optimal consumption derived in the proof of existence is not constructive. Drawing on the discussion in \cite{BR}, we conjecture that the optimal policy involves consuming just enough to maintain the associated level of satisfaction above a certain threshold, referred to as the minimal level of satisfaction. Unlike the time-additive case, the recursive nature of the utility introduces a key difference: the minimal level of satisfaction is determined by a fully-coupled forward-backward system that arises from the first-order conditions. This system serves a similar function as the Hamilton-Jacobi-Bellman equation in the dynamic programming framework does.
 
 %Therefore, in our model, we do not need to restrict ourselves to the Markovian setting.

    Another contribution of the paper is the analysis of the sample paths of optimal consumption plans. We consider the following two types of consumption behavior: lump-sum consumption and consumption in rates. Recalling that in the deterministic case studied in \cite{LRY}, for $t$ in a free interval, we have 
    \begin{align*}
        \exp\left(\int_0^s \partial_u f(r,Y_r^C,U_r^C)dr\right)\partial_y f(s,Y_s^C,U_s^C)=M\frac{r+\beta}{\beta}e^{-rt},
    \end{align*}
    where $M$ is the Lagrange multiplier, $r$ is the interest rate and $\beta$ is the depreciation rate. A similar relation still holds for the stochastic case with $e^{-rt}$ of the right-hand side being replaced by the state price density. When there is lump consumption at $t$, the above equality becomes an inequality. We show that lump-sum consumption can only occur when a surprise happens, which means that at that point, the filtration is not left-continuous or there is a jump of the state price density. With the help of the quantitative properties, we provide a more detailed structure of the optimal consumption plan when the state price density is characterized by a geometric Poisson process.

	The paper is organized as the follows. In Section 2, we   formulate the recursive utility   in detail and study its properties.   Section 3 is devoted to well-posedness of the utility maximization problem and derives the necessary and sufficient condition for optimality when the financial market is complete. In Section 4, we investigate the construction and some properties of the optimal consumption plan. Finally, we study a typical case where the state price density is given by a geometric Poisson process.%, we provide a more accurate description of the structure of optimal consumption plans.

\section{Stochastic Differential Utility with Intertemporal Substitution}

Let   $(\Omega,\mathcal{F},\P_0)$ be  a probability space carrying     a standard Brownian motion $B=\{B_t\}_{t\in[0,T]}$ and    a Poisson random measure $v$ on $(\mathbb{R}_*,\mathfrak{B}(\mathbb{R}_*))$ with intensity $\vartheta$, where $\mathbb{R}_*=\mathbb{R}\backslash \{0\}$. The associated compensated random measure is denoted by $\tilde{v}$, i.e., $\tilde{v}(dt,dx)=v(dt,dx)-dt\vartheta(dx)$. Denote by $(\mathcal{F}_t)_{t\in[0,T]}$ the filtration  generated by $B,v$ and the class of $\P_0$-negligible sets. We assume that $\vartheta$ is a $\sigma$-finite measure on $\mathbb{R}_*$ such that $\int_{\mathbb{R}_*}(1\wedge x^2)\vartheta(dx)<\infty$. For  more details,   refer to \cite{CT,Delong,KS} and the references therein. 

  Let the interest rate $r=\{r_t\}_{t\in[0,T]}$ be a bounded, progressively measurable process and denote by $\gamma_t=\exp(-\int_0^t r_sds)$ the discount factor.  Let $\mathcal{P}$ be a collection of $\P_0$-equivalent probability measures $\P^*$ on $(\Omega,\mathcal{F}_T)$. 
  Denote by $\E^*[\cdot]$ (resp. $\E[\cdot]$)   the expectation   under probability $\P^*$ (resp. $\P_0$).
  
Let $\mathcal{X}$ be  the set of consumption plans, consisting   of   distribution functions of nonnegative optional random measures.   An agent with initial wealth $w> 0$ faces the budget constraint 
\begin{displaymath}
\mathcal{A}(w)=\left\{C\in\mathcal{X}\,|\, \sup_{\P^*\in\mathcal{P}}\E^*\left[\int_0^T \gamma_tdC_t\right]\leq w\right\}.
\end{displaymath}

Let us now develop the model  for our preferences that extends the Hindy-Huang-Kreps (\cite{HHK}) approach to recursive utilities in our jump--diffusion setting.
For a fixed consumption plan $C\in\mathcal{X}$, the agent's level of satisfaction at time $t$ is given by
\begin{displaymath}
Y_t^C=\eta_t+\int_0^t \theta_{t,s}dC_s,
\end{displaymath}
where  $\eta:[0,T]\rightarrow [0,\infty)$ and $\theta:[0,T]^2\rightarrow(0,\infty)$ are continuous deterministic functions.   The quantity $\theta_{t,s}$ can be seen as the weight assigned at time $t$ to consumption made at time $s\leq t$ and $\eta_t$ describes an exogenous level of satisfaction for time $t$.   
We model the  recursive utility of the agent according to the following recursive equation
\begin{equation}\label{e3}
U_t^C=\E_t\left[\int_t^T f(s,Y_s^C,U_s^C)ds\right]
\end{equation}
for some aggregator $f:[0,T]\times\mathbb{R}_+\times\mathbb{R}\rightarrow\mathbb{R}$. 
We first need to show that such preferences are well-defined.

% which is increasing and concave in its second component.
\begin{assumption}\label{a1}
	\begin{description}
		\item[(i)] For any $s\in[0,T]$, $f(s,\cdot,\cdot)$ is strictly concave and continuous differentiable, and for any $s\in[0,T]$, $u\in\mathbb{R}$, $f(s,\cdot,u)$ is strictly increasing.
		\item[(ii)] For any $(s,y)\in[0,T]\times\mathbb{R}_+$, $u,u'\in\mathbb{R}$, there exists a constant $K$ such that
		\[|f(s,y,u)-f(s,y,u')|\leq K|u-u'|.\]
		\item[(iii)] For any $s\in[0,T]$, there exist two constants $\alpha\in(0,\frac{1}{2})$ and $K$ such that
		\[|f(s,y,0)|\leq K(1+|y|^\alpha).\]
         \item[(iv)] There exist a probability measure $\P^*\in \mathcal{P}$ and a constant $p>1$ with $2\alpha p<1$, such that $\frac{d\P_0}{d\P^*}\in L^q(\P^*)$ for $q=\frac{1}{1-2\alpha p}$.
	\end{description}
\end{assumption}

We   provide some examples of the probability measure $\P^*$ satisfying Assumption \ref{a1} (iv).

\begin{example}
 
\begin{description}
\item[(i)] For $\theta\neq 0$, if the density process takes the following form
	\begin{displaymath}
		\frac{d\P^*}{d\P_0}=\exp\left(\theta B_T-\frac{1}{2}\theta^2 T \right),
	\end{displaymath}
	then for any $p\geq 1$, $\frac{d\P_0}{d\P^*}\in L^p(\P^*)$. Indeed, it is easy to check that
	\begin{align*}
		\E^*\left[\left(\frac{d\P_0}{d\P^*}\right)^p\right]&=\E\left[\left(\frac{d\P^*}{d\P_0}\right)\left(\frac{d\P_0}{d\P^*}\right)^p\right]=\E\left[\exp\left(\theta(1-p) B_T-\frac{1}{2}\theta^2(1-p) T\right)\right]\\
  &=\E\left[\exp\left(\theta(1-p) B_T-\frac{1}{2}\theta^2(1-p)^2 T+\frac{1}{2}\theta^2p(p-1)T \right)\right]<\infty.
	\end{align*}

\iffalse
(i)	Let $\zeta=\{\zeta_t\}_{t\in[0,T]}$ be a bounded, progressively measurable process. If the density process takes the following form
	\begin{displaymath}
		\frac{d\P^*}{d\P_0}=\exp\left(\int_0^T \zeta_t dB_t-\frac{1}{2}\int_0^T \zeta_t^2 dt\right),
	\end{displaymath}
	then for any $p\geq 1$, $\frac{d\P_0}{d\P^*}\in L^p(P^*)$. Indeed, note that for any $p\geq 1$, $\xi$ is bounded, where
	\begin{displaymath}
		\xi=\exp\left(\int_0^T \frac{1}{2}(p^2-3p)\zeta_t^2dt\right).
	\end{displaymath}
	It is easy to check that
	\begin{align*}
		\E^*\left[\left(\frac{d\P_0}{d\P^*}\right)^p\right]=\E\left[\left(\frac{d\P^*}{d\P_0}\right)\left(\frac{d\P_0}{d\P^*}\right)^p\right]
=\E\left[\xi\exp\left(\int_0^T (1-p)\zeta_tdB_t-\frac{1}{2}\int_0^T (1-p)^2\zeta_t^2dt\right)\right]<\infty.
		%\leq &C\left(\E\left[\exp\left(\int_0^T 2(1-p)\theta_tdB_t-\frac{1}{2}\int_0^T 4(1-p)^2\theta_t^2dt\right)\right]\right)^{\frac{1}{2}}<\infty.
	\end{align*}
 \fi

\item[(ii)] If $\{N_t\}_{t\in[0,T]}$ is a Poisson process with intensity $\lambda$, for $A\in \mathfrak{B}(\mathbb{R}_*)$, 
\begin{align*}
    v(t,A):=\sum_{0<s\leq t}I_{A}(\Delta N_s)
\end{align*}
defines a Poisson random measure.
In this case, $v(dt,dx)=dN_t\delta_{1}(x)$ and $\vartheta(dx)=\lambda\delta_{1}(x)$, where $\delta_{1}$ is the Dirac measure. For $\theta\neq 0$, if the density process takes the following form
	\begin{displaymath}
		\frac{d\P^*}{d\P_0}=\exp\left(\theta N_T-\lambda(e^\theta-1) T \right),
	\end{displaymath}
	then for any $p\geq 1$, $\frac{d\P_0}{d\P^*}\in L^p(\P^*)$. Actually, simple calculation yields that 
 \begin{align*}
		&\E^*\left[\left(\frac{d\P_0}{d\P^*}\right)^p\right]=\E\left[\left(\frac{d\P^*}{d\P_0}\right)\left(\frac{d\P_0}{d\P^*}\right)^p\right]=\E\left[\exp\left(\theta(1-p) N_T-\lambda(e^\theta-1)(1-p) T\right)\right]\\
  =&\E\left[\exp\left(\theta(1-p) N_T-\lambda(e^{\theta(1-p)}-1)T+\lambda(e^{\theta(1-p)}-1)T-\lambda(e^\theta-1)(1-p) T\right)\right]<\infty.
	\end{align*}
 \end{description}
\end{example}

\begin{remark}\label{BSDEforU}
    The utility $U^C$ induced by \eqref{e3} can be viewed as the first component of a backward stochastic differential equation (BSDE) with jumps taking the following form
    \begin{equation}\label{BSDE}
        U_t^C=\int_t^T f(s,Y^C_s,U^C_s)ds-\int_t^T \bar{Z}_s dB_s-\int_t^T\int_{\mathbb{R}_*}\bar{\Psi}_s(x)\tilde{v}(dt,dx).
    \end{equation}
    In the setting of BSDE with jumps, the function $f$ is usually called a generator, which can also depend on $\bar{Z}$ and $\bar{\Psi}$. The solution to BSDE \eqref{BSDE} is a triple of processes $(U^C,\bar{Z},\bar{\Psi})\in \mathbb{S}^2\times \mathbb{H}^2\times \mathbb{H}^2_v$, where
    \begin{itemize}
        \item $\mathbb{S}^2$ is the space of adapted, c\`{a}dl\`{a}g processes $\zeta:\Omega\times[0,T]\rightarrow \mathbb{R}$ satisfying
    \begin{align*}
        \E\left[\sup_{t\in[0,T]}|\zeta_t|^2\right]<\infty;
    \end{align*}
    \item $\mathbb{H}^2$ is the space of predictable processes $Z:\Omega\times[0,T]\rightarrow \mathbb{R}$ satisfying
    \begin{align*}
        \E\left[\int_0^T |Z_t|^2dt\right]<\infty;
    \end{align*}
    \item $\mathbb{H}^2_v$ is the space of predictable processes $\Psi:\Omega\times[0,T]\times \mathbb{R}\rightarrow \mathbb{R}$ satisfying
    \begin{align*}
        \E\left[\int_0^T \int_{\mathbb{R}_*}|\Psi_t(x)|^2\vartheta(dx)dt\right]<\infty.
    \end{align*}
    \end{itemize}
    For more details about BSDE driven by Brownian motion and Poisson random measure, we refer to \cite{BBP, Delong,Royer,TL} and the references therein.
\end{remark}

Stochastic differential utility is well understood when consumption occurs in rates, see, e.g. \cite{DE} and 
\cite{KS}. In the following, we show the established  results carry over when the utility is induced by the level of satisfaction instead of the consumption rate.

\begin{theorem}\label{propertyofutility}
    Under Assumption \ref{a1}, for any $w'>0$ and $C\in \mathcal{A}(w')$, there exists a unique solution $(U^C,\bar{Z},\bar{\Psi})\in \mathbb{S}^2\times \mathbb{H}^2\times \mathbb{H}^2_v$ to \eqref{BSDE}.
    
    The   following properties hold true.
    \begin{itemize}
        \item [1.] Suppose that $\{C^n\}_{n=1}^\infty\subset \mathcal{A}(w')$ converges almost surely to $C\in\mathcal{A}(w')$ in the weak topology of measures on $[0,T]$. Let $Y$ (resp. $Y^n$) and $U$ (resp. $U^n$) be the level of satisfaction and utility induced by $C$ (resp. $C^n$). Then, we have $U^n_t\rightarrow U_t$, $t\in[0,T]$.
        \item [2.] The utility is (strictly) concave in $C$. Mathematically, let $C^i\in \mathcal{A}(w')$, $i=1,2$. For any $\lambda\in(0,1)$, set $C=\lambda C^1+(1-\lambda)C^2$. Let $U$ and $U^i$ be the utility correspond to $C$ and $C^i$, $i=1,2$, respectively. Then,  we have $U_t\geq \lambda U^1_t+(1-\lambda)U^2_t$, $t\in[0,T]$. Moreover, if $C^1$ and $C^2$ are not indistinguishable, we have $U_0> \lambda U^1_0+(1-\lambda)U^2_0$. 
        \item[3.] The utility is monotone in $C$. That is, let $C^i\in \mathcal{A}(w')$, $i=1,2$ and $Y^i$ be the utility correspond to $C^i$ such that $Y^1_t\leq Y^2_t$, $t\in[0,T]$. Then, we have $U^1_t\leq U^2_t$, $t\in[0,T]$.
    \end{itemize}
\end{theorem}

\begin{proof}
 By assumption, the processes $r,\eta,\theta$  are bounded by some constant $M$. In this proof, we always denote by $L$ a positive constant depending on $\alpha,p,K,w',T,M$, which may vary from line to line. We first prove the existence and uniqueness of $U^C$ for any given $C\in\mathcal{A}(w')$. By Theorem 3.1.1 in \cite{Delong}, it suffices to prove that for any $C\in\mathcal{A}(w')$, we have 
    \begin{align*}
        \E\left[\int_0^T |f(t,Y^C_t,0)|^2 dt\right]<\infty.
    \end{align*}
By Assumption \ref{a1} and Lemma \ref{l1} (i), we have
    \begin{align*}
        |f(t,Y^C_t,0)|^2\leq L(1+|C_t|^{2\alpha})\leq L(1+|C_T|^{2\alpha}).
    \end{align*}
    By Lemma \ref{l1} (iii), we obtain the desired result. 
    
    We next claim that for any $t\in[0,T]$,  we have
		\begin{equation}\label{claim1}
	 \E\left[|U_t^C|^{2p}\right]<\infty,
		\end{equation}
   where $p>1$ is given in Assumption \ref{a1} (iv). A simple calculation yields that
	\begin{align*}
		|U_t^C|^{2p}&\leq L\left\{\E_t\left[\bigg|\int_t^T f(s,Y_s^C,0)ds\bigg|^{2p}\right]+\E_t\left[\bigg|\int_t^T (f(s,Y_s^C,0)-f(s,Y_s^C,U_s^C))ds\bigg|^{2p}\right]\right\}\\
		&\leq L \E_t\left[\int_t^T \left(1+|Y_s^C|^{2\alpha p}\right)ds\right]+L \E_t\left[\int_t^T |U_s^C|^{2p} ds\right]\\
		&\leq L \E_t\left[1+|C_T|^{2\alpha p}\right]+L \int_t^T\E_t\left[|U_s^C|^{2p}\right]ds.
	\end{align*}
  Taking expectations on both sides, we obtain that
	\begin{align*}\label{e2}
		\E[|U_t^C|^{2p}]\leq L \E[1+|C_T|^{2\alpha p}]+L\int_t^T \E[|U_s^C|^{2p}]ds.
	\end{align*}
	 Applying the Gronwall inequality, the estimate \eqref{claim1} follows.

   In the following, we show the continuity of the utility function.  By Lemma 3.1.1 in \cite{Delong}, we have 
    \begin{align*}
        \E\left[\sup_{t\in[0,T]}|U_t-U^n_t|^2\right]\leq L\E\left[\int_0^T|f(t,Y_t,U_t)-f(t,Y^n_t,U_t)|^2dt\right].
    \end{align*}
    Set $\xi^n=\int_0^T|f(t,Y_t,U_t)-f(t,Y^n_t,U_t)|^2dt$. It is easy to check that
    \begin{displaymath}
	|\xi_n|^p\leq L\int_0^T (1+|U_t|^{2p}+|Y_t^n|^{2\alpha p}+|Y_t|^{2\alpha p})dt\leq L\int_0^T (1+|U_t|^{2p}+|{{C}^n_t}|^{2\alpha p}+|{C_t}|^{2\alpha p})dt.
	\end{displaymath}
 By Lemma \ref{l1} (iii) and \eqref{claim1}, for any $n=1,2,\cdots$, we have $\E[|\xi^n|^p]<\infty$, which implies that $\{\xi^n\}_{n=1}^\infty$ is uniformly integrable. By Lemma \ref{l1} (ii), $\xi^n\rightarrow 0$ as $n$ goes to infinity. The above analysis indicates that 
 \begin{align*}
     \lim_{n\rightarrow \infty}\E\left[\sup_{t\in[0,T]}|U_t-U^n_t|^2\right]\leq \lim_{n\rightarrow \infty}L\E[\xi^n]=0.
 \end{align*}

 Then, we prove the concavity of the utility function. Let $Y$ and $Y^i$ be the level of satisfaction correspond to $C$ and $C^i$, $i=1,2$, respectively. Since $Y=\lambda Y^1+(1-\lambda)Y^2$, $U$ can be seen as the solution to the BSDE with generator $g$, where
 \begin{align*}
     g(s,u)=f(s,\lambda Y^1_s+(1-\lambda)Y^2_s,u).
 \end{align*}
 Set $\tilde{U}=\lambda U^1+(1-\lambda)U^2$. It is easy to check that $\tilde{U}$ can be regarded as the solution to the BSDE with generator $\tilde{g}$, where 
 \begin{align*}
     \tilde{g}(s,u)=&f(s,\lambda Y^1_s+(1-\lambda)Y^2_s,u)+\lambda f(s,Y^1_s,U^1_s)\\
     &+(1-\lambda)f(s,Y^2_s,U^2_s)-f(s,\lambda Y^1_s+(1-\lambda)Y^2_s,\lambda U^1_s+(1-\lambda)U^2_s).
 \end{align*} 
 Noting that $g\geq \tilde{g}$, by Theorem 3.2.2   in \cite{Delong}, we have $U_t\geq \tilde{U}_t=\lambda U^1_t+(1-\lambda)U^2_t$. 
 
 For the strict concavity, suppose otherwise that $U_0=\tilde{U}_0$. It follows from Theorem 3.2.2 in \cite{Delong} that $U_t=\tilde{U}_t(=\lambda U^1+(1-\lambda)U^2)$, $t\in[0,T]$. Taking derivatives and comparing the $dt$-term, we obtain that
 \begin{equation}\label{contradiction}
     f(t,\lambda Y^1_t+(1-\lambda)Y^2_t,U_t)=\lambda f(t,Y^1_t+U^1_t)+(1-\lambda)f(t,Y^2_t,U^2_t), \ \forall t\in[0,T].
 \end{equation}
However, since $C^1$ and $C^2$ are not indistinguishable, by a similar analysis as the proof of Theorem 2.3 in \cite{BR}, $Y^1$ and $Y^2$ differ on an open time interval on a set with positive probability, which contradicts \eqref{contradiction} due to the strict concavity of $f$. 

The monotonicity of the utility is a direct consequence of Theorem 3.2.2 in \cite{Delong}. 
\end{proof}

\begin{remark}
    An inspection of the proof of Theorem \ref{propertyofutility} shows that  existence and uniqueness as well as the continuity property for the recursive utility still hold if we drop condition (i) in Assumption \ref{a1}. Strict concavity does not require that $f$ be strictly increasing in $y$-component. The strict  monotonicity  is used below   to ensure that the agent's utility function is nonsatiated, which implies that the agent will always exhaust his budget.
\end{remark}

\section{The Optimal Consumption Problem: Well--Posedness and First-Order Conditions}

In the remainder of the paper, we study optimal consumption choice for intertemporal preferences that are both recursive and account for local substitution, i.e.
\begin{equation}\label{e1}
v(w)=\sup_{C\in\mathcal{A}(w)}U_0^C.%=\sup_{C\in\mathcal{A}_g(\omega)}\mathcal{E}^{g}[U(C)].
\end{equation}
We first show that the problem is well-posed.

    \begin{theorem}\label{t1}
	Under Assumption \ref{a1}, the utility maximization problem \eqref{e1} has a solution. Moreover, if  $C\rightarrow Y^C$ is injective up to $P_0$-indistinguishability, the solution is unique.
\end{theorem}

 Note that the mapping $C\rightarrow Y^C$ is injective 
	if the function $\theta$ is separable, i.e., there exist two strictly positive, continuous functions $\theta^1,\theta^2:[0,T]\times\mathbb{R}$ such that $\theta_{t,s}=\theta^1_t \theta^2_s$.

\begin{proof}
Uniqueness follows from the strict concavity of the utility function. In fact, suppose that $C^i$ are optimal for problem \eqref{e1}, $i=1,2$ and they are not indistinguishable. Set $C=\frac{1}{2}(C^1+C^2)$. It is easy to check that $C\in\mathcal{A}(w)$. However, we have 
\begin{align*}
    U^C_0>\frac{1}{2}(U_0^{C^1}+U_0^{C^2})=\sup_{C\in\mathcal{A}(w)}U_0^C,
\end{align*}
which is a contradiction.
%	We first prove the uniqueness. Suppose that $C^i$ are optimal to problem \eqref{e1} and they are not indistinguishable,  then the level of satisfaction $Y^i=Y^{C^i}$ are not indistinguishable, $i=1,2$. Set $C=\frac{1}{2}(C^1+C^2)$ and $Y=Y^C$. It is easy to check that $C\in\mathcal{A}(w)$ and $Y=\frac{1}{2}(Y^1+Y^2)$. By Proposition \ref{p1}, we have $U_t^C\geq \frac{1}{2}(U_t^{C^1}+U_t^{C^2})$. Similar analysis as the proof of Theorem 2.3 in \cite{BR}, $Y^1$ and $Y^2$ differ on an open time interval on a set with positive probability. By the strict concavity and the increasing property of $f(t,\cdot,\cdot)$, we have
%	\begin{align*}
%	U_0^C&=\E\left[\int_0^T f(t,Y_t,U_t^C)dt\right]\\
%	&>\E\left[\int_0^T \frac{1}{2}(f(t,Y^1_t,U_t^{C^1})+f(t,Y^2_t,U_t^{C^2}))dt\right]\\
%	&= \frac{1}{2}\left(\E\left[\int_0^T f(t,Y^1_t,U_t^{C^1})dt\right]+\E\left[\int_0^T f(t,Y^2_t,U_t^{C^2})dt\right]\right)\\
%	&=\frac{1}{2}(U_0^{C^1}+U_0^{C^2})=\sup_{C\in\mathcal{A}(w)}U_0^C,
%	\end{align*}
%	which is a contradiction.
	
	We are now in a position to show existence. Choose a maximizing sequence $\{C^n\}_{n=1}^\infty\subset \mathcal{A}(w)$ such that
	\begin{displaymath}
	\sup_{C\in\mathcal{A}(w)}U_0^C=\lim_{n\rightarrow\infty}U_0^{C^n}.
	\end{displaymath}
	By Lemma \ref{l1} (iii) and Kabanov's version of Koml\'{o}s' theorem (see \cite{K99}), there exists a subsequence, for simplicity, still denoted by $\{C^n\}_{n=1}^\infty$, such that
	\begin{equation}\label{convergene}
	\widetilde{C}^n_t:=\frac{1}{n}\sum_{k=1}^{n}C^k_t\rightarrow C^*_t, \ \textrm{as }n\rightarrow\infty
	\end{equation}
	for $t=T$ and for every point of continuity $t$ of $C^*$. We claim that $\{\widetilde{C}^n\}_{n=1}^\infty$ is also a maximizing sequence for problem \eqref{e1}. First, it is easy to check that $\widetilde{C}^n\in \mathcal{A}(w)$, for any $n\in\mathbb{N}$. Consequently, we have $U^{\widetilde{C}^n}_0\leq \sup_{C\in\mathcal{A}(w)}U^C_0$. By Theorem \ref{propertyofutility}, we obtain that %Indeed, the convexity of $\widetilde{\mathcal{E}}^g[\cdot]$ implies that $\widetilde{C}^n\in \mathcal{A}(\omega)$, for any $n\in\mathbb{N}$. Therefore, we have $V(\widetilde{C}^n)\leq \sup_{C\in\mathcal{A}(\omega)}V(C)$. On the other hand, it is easy to check that
	\begin{displaymath}
	U_0^{\widetilde{C}^n}\geq \frac{1}{n}\sum_{k=1}^{n}U_0^{C^k}.
	\end{displaymath}
	It follows that
	\begin{displaymath}
		\liminf_{n\rightarrow\infty}U_0^{\widetilde{C}^n}\geq \liminf_{n\rightarrow\infty}\frac{1}{n}\sum_{k=1}^{n}U_0^{C^k}=\sup_{C\in\mathcal{A}(w)}U_0^C.
	\end{displaymath}
	Hence, $\{\widetilde{C}^n\}_{n=1}^\infty$ is a maximizing sequence for problem \eqref{e1}. We then show that ${C}^*$ is optimal  for problem \eqref{e1}. Note that $\gamma$ is continuous, which implies that
	\begin{displaymath}
		\lim_{n\rightarrow\infty} \int_0^T \gamma_td\widetilde{C}^n_t=\int_0^T \gamma_td{C}^*_t.
	\end{displaymath}
	By Fatou's lemma, we have for any $\P^*\in\mathcal{P}$,
	\begin{displaymath}
		\E^*\left[\int_0^T \gamma_t dC^*_t\right]\leq \liminf_{n\rightarrow\infty} \E^*\left[\int_0^T \gamma_td\widetilde{C}^n_t\right]\leq w,
	\end{displaymath}
	which indicates that $C^*\in\mathcal{A}(w)$. Recalling \eqref{convergene} and using Theorem \ref{propertyofutility} again, we have
\begin{displaymath}
		U^{C^*}_0=\lim_{n\rightarrow\infty} U^{\widetilde{C}^n}_0=\sup_{C\in\mathcal{A}(w)}U_0^C.
	\end{displaymath}
 Thus, $C^*$ is an optimal budget-feasible consumption plan.
 \iffalse By Assumption \ref{a1}, we obtain that
	\begin{displaymath}
		\xi_n:=\int_0^T |f(t,Y^{\widetilde{C}^n}_t,U^{C^*}_t)-f(t,Y^{C^*}_t,U^{C^*}_t)|dt\rightarrow 0, \P_0\textrm{-}a.s.
	\end{displaymath}
	We claim that $\{\xi^n,n\in\mathbb{N}\}$ is uniformly integrable. Indeed, it is easy to check that for any $p>1$ with $\alpha p<1$, we have
	\begin{displaymath}
	|\xi_n|^p\leq L\int_0^T (1+|U_t^{C^*}|^p+|Y_t^{\widetilde{C}^n}|^{\alpha p}+|Y_t^{C^*}|^{\alpha p})dt\leq L\int_0^T (1+|U_t^{C^*}|^p+|{\widetilde{C}^n_t}|^{\alpha p}+|{C^*_t}|^{\alpha p})dt.
	\end{displaymath}
	By Lemma \ref{l1}, we get that $\E[|\xi^n|^p]<\infty$. Hence, the family $\{\xi^n,n\in\mathbb{N}\}$ is uniformly integrable, together with Proposition \ref{p1}, which yields that
	\begin{displaymath}
		U^{C^*}_0=\lim_{n\rightarrow\infty} U^{\widetilde{C}^n}_0.
	\end{displaymath}
	Thus, $C^*$ is an optimal budget-feasible consumption plan.
 \fi
\end{proof}

%\section{First-order conditions for optimality}	

In the next step, we establish a version of the Kuhn-Tucker theorem for our utility maximization problem \eqref{e1} in complete markets.  The main mathematical tool to prove the first-order conditions is the stochastic Gronwall inequality (see Proposition B.1 in \cite{KSS}).  

\begin{assumption}\label{a2}
	The financial market is complete in the sense that $\mathcal{P}$ is a singleton. We denote this probability measure by $\P^*$, and impose Assumption \ref{a1} (iv).
\end{assumption}
	
We write 
\begin{displaymath}
	\psi_t=\gamma_t \frac{d\P^*}{d\P_0}\Big|_{\mathcal{F}_t}, \ t\in[0,T],
\end{displaymath}	
for the state-price density 
and
\begin{equation}\label{nablaV}
	\nabla V(C)(t)=\E_t\left[\int_t^T \exp\left(\int_0^s \partial_u f(r,Y_r^C,U_r^C)dr\right)\partial_y f(s,Y_s^C,U_s^C)\theta_{s,t}ds\right].
\end{equation}
 for the marginal   utility that is derived from an additional infinitesimal consumption at time $t$, otherwise following the consumption $C$. From the mathematical point of view, $\nabla V(C)$ is the Riesz representation of the utility gradient at $C$.

\begin{remark}\label{r2}
\begin{description}
    \item[(i)] By Theorem (1.33) in Jacod \cite{J79}, for any $C'\in\mathcal{X}$, we have the following identity
	\begin{displaymath}
	     \E\left[\int_0^T \nabla V(C)(t)dC'_t\right]=\E\left[\int_0^T \Phi(t)dC'_t\right],
	\end{displaymath}
	where $\Phi$ is the nonnegative, product-measurable process
	\begin{displaymath}
		\Phi(t,\omega)=\int_t^T \exp\left(\int_0^s \partial_u f(r,Y_r^{C(\omega)},U_r^C(\omega))dr\right)\partial_y f(s,Y_s^{C(\omega)},U_s^C(\omega))\theta_{s,t}ds.
	\end{displaymath}
 \item[(ii)] Let  $f(t,y,u)=g(t,y)-\delta u$, where $\delta$ is a constant and $g$ is a continuous function which is increasing and concave in its second argument. Then  the solution to Eq. \eqref{e3} is given by 
	\begin{displaymath}
		U_t^C=\E_t\left[\int_t^T e^{-\delta(s-t)} g(s, Y_s^C)ds\right].
	\end{displaymath}
	In this case, the problem \eqref{e1} degenerates to the separable utility maximization problem studied in \cite{BR}, and we have
 \begin{displaymath}
	\nabla V(C)(t)=\E_t\left[\int_t^T e^{-\delta s}\partial_y g(s,Y_s^C)\theta_{s,t}ds\right],
\end{displaymath}
which coincides with Equation (5) in \cite{BR}.
\end{description}	
	%Hence, $\nabla V(C)$ can also be defined as the optional projection of $\Phi$.
\end{remark}

\begin{theorem}\label{t2}
	Under Assumption \ref{a1} and Assumption \ref{a2}, a consumption plan $C^*\in\mathcal{X}$ solves the utility maximization problem \eqref{e1} if and only if the following conditions hold true for some finite Lagrange multiplier $M>0$:
	\begin{description}
		\item[(i)] $\E\left[\int_0^T \psi_t dC^*_t\right]=w$.
		\item[(ii)] $\nabla V(C^*)(t)\leq M\psi_t$ for any $t\in[0,T]$, $\P_0$-a.s.
		\item[(iii)] $\E\left[\int_0^T (\nabla V(C^*)(t)-M\psi_t)dC^*_t\right]=0$, which means that, for almost all $\omega\in\Omega$, $C^*(\omega)$ is flat off the set
		\begin{displaymath}
			\{t\in[0,T]|\nabla V(C^*)(t,\omega)=M\psi_t(\omega)\}.
		\end{displaymath}
	\end{description}
\end{theorem}

\begin{proof}
	We first prove the sufficiency. Let $C^*\in\mathcal{X}$ be the consumption plan satisfying conditions (i)-(iii) and $C\in\mathcal{X}$ be another budget-feasible consumption plan. For simplicity, set $Y^*=Y^{C^*}$, $Y=Y^C$, $U^*=U^{C^*}$, $U=U^C$, $\hat{U}=U^*-U$, $\hat{Y}=Y^*-Y$ and $\hat{C}=C^*-C$. By the concavity of the felicity function $f$, it is easy to check that, for any $t\in[0,T)$ and stopping time $\tau$
	\begin{align*}
		\hat{U}_tI_{\{\tau>t\}}=& \E_t\left[I_{\{\tau>t\}}\int_t^\tau \left(f(s,Y_s^*,U_s^*)-f(s,Y_s,U_s)\right)ds+I_{\{\tau>t\}}\hat{U}_\tau\right]\\
		\geq &\E_t\left[I_{\{\tau>t\}}\int_t^\tau \left(\partial_y f(s,Y_s^*,U_s^*)\hat{Y}_s+ \partial_u f(s,Y_s^*,U_s^*)\hat{U}_s\right)ds+I_{\{\tau>t\}}\hat{U}_\tau\right]\\
		=&\E_t\left[I_{\{\tau>t\}}\int_t^\tau \left(\partial_y f(s,Y_s^*,U_s^*)\int_0^s \theta_{s,t}d\hat{C}_t+ \partial_u f(s,Y_s^*,U_s^*)\hat{U}_s\right)ds+I_{\{\tau>t\}}\hat{U}_\tau\right].
	\end{align*}
	By Proposition B.1 in \cite{KSS}, we obtain that
	\begin{displaymath}
		U_0^*-U_0\geq \E\left[\int_0^T \exp\left(\int_0^s \partial_uf(r,Y^*_r,U^*_r)dr\right)\partial_yf(s,Y_s^*,U_s^*)\int_0^s \theta_{s,t}(dC_t^*-dC_t)ds\right].
	\end{displaymath}
	We divide the last expectation into two terms:
	\begin{displaymath}
		I^*=\E\left[\int_0^T \exp\left(\int_0^s \partial_uf(r,Y^*_r,U^*_r)dr\right)\partial_yf(s,Y_s^*,U_s^*)\int_0^s \theta_{s,t}dC_t^*ds\right]
	\end{displaymath}
	and
		\begin{displaymath}
	I=\E\left[\int_0^T \exp\left(\int_0^s \partial_uf(r,Y^*_r,U^*_r)dr\right)\partial_yf(s,Y_s^*,U_s^*)\int_0^s \theta_{s,t}dC_tds\right].
	\end{displaymath}
	By simple calculation, we derive that
	\begin{align*}
		I^*&=\E\left[\int_0^T \int_t^T \exp\left(\int_0^s \partial_uf(r,Y^*_r,U^*_r)dr\right)\partial_yf(s,Y_s^*,U_s^*)\theta_{s,t}ds dC_t^*\right]\\
		&=\E\left[\int_0^T \E_t\left[\int_t^T \exp\left(\int_0^s \partial_uf(r,Y^*_r,U^*_r)dr\right)\partial_yf(s,Y_s^*,U_s^*)\theta_{s,t}ds\right] dC_t^*\right]\\
		&=\E\left[\int_0^T \nabla V(C^*)(t)dC_t^*\right]=M E\left[\int_0^T \psi_tdC_t^*\right]=Mw,
	\end{align*}
	where we use the Fubini theorem in the first equality, Remark \ref{r2} in the second equality. A similar analysis yields that
	\begin{displaymath}
		I=\E\left[\int_0^T \nabla V(C^*)(t)dC_t\right]\leq M \E\left[\int_0^T \psi_tdC_t\right]\leq Mw.
	\end{displaymath}
	Combining the above results implies that $U^*_0-U_0\geq 0$. Hence, $C^*$ is optimal. Necessity follows from Lemma \ref{l4} below and Lemma 3.4 in \cite{BR}.
\end{proof}

\section{Properties and Construction of the Optimal Consumption Plan}

In this section, we investigate properties of the optimal consumption plan, proving first the the dynamic programming principle. We then conduct a quantitative analysis of consumption in gulps versus consumption in rates. Furthermore, we provide a more detailed construction of the optimal consumption plan using first-order conditions.

\subsection{Dynamic programming principle}
Theorem \ref{t2} allows us to  establish a version of the dynamic programming principle. If a consumption plan is optimal at time zero, it remains the best possible continuation at any later time.
\begin{theorem}\label{t3}
	Let $\tau\leq T$ be a stopping time and set
	\begin{displaymath}
		\mathcal{A}_\tau=\left\{C\in\mathcal{X}\Bigg| C=C^* \textrm{ on } [0,\tau) \textrm{ and }\E_\tau\left[\int_\tau^T \psi_t dC_t\right]\leq \E_\tau\left[\int_\tau^T \psi_t dC^*_t\right]\right\}.%\Gamma_\tau(C)\leq \Gamma_\tau(C^*)\},
	\end{displaymath}
%	where
%	\begin{displaymath}
%		\Gamma_\tau(C)=\frac{1}{\psi_\tau}\E_\tau\left[\int_\tau^T \psi_t dC_t\right],
%	\end{displaymath}
%is the price functional at time $\tau$. 
If $C^*$ is optimal for the problem \eqref{e1}, it also solves the problem
	\begin{displaymath}
		\esssup_{C\in\mathcal{A}_\tau}U_\tau^C.
	\end{displaymath}
	%It means that an optimal consumption plan at time zero is its best continuation at any time afterward.
\end{theorem}

\begin{proof}
	For any $C\in\mathcal{A}_\tau$, let $U^*=U^{C^*}$, $U=U^C$, $Y^*=Y^{C^*}$ and $Y=Y^C$. By a similar analysis as the proof of Theorem \ref{t2}, we have
	\begin{displaymath}
		U_\tau^*-U_\tau\geq \E_\tau\left[\int_\tau^T \exp\left(\int_\tau^t \partial_uf(r,Y_r^*,U_r^*)dr\right)\partial_yf(t,Y_t^*,U_t^*)(Y_t^*-Y_t)dt\right].
		\end{displaymath}
	Multiplying $\exp(\int_0^\tau \partial_uf(r,Y_r^*,U_r^*)dr)$ on both sides yields that
	\begin{align*}
		&\exp\left(\int_0^\tau \partial_uf(r,Y_r^*,U_r^*)dr\right)(U_\tau^*-U_\tau)\\
		\geq& \E_\tau\left[\int_\tau^T \exp\left(\int_0^t \partial_uf(r,Y_r^*,U_r^*)dr\right)\partial_yf(t,Y_t^*,U_t^*)(Y_t^*-Y_t)dt\right]\\
		=& \E_\tau\left[\int_\tau^T \exp\left(\int_0^t \partial_uf(r,Y_r^*,U_r^*)dr\right)\partial_yf(t,Y_t^*,U_t^*)\int_\tau^t \theta_{t,s}(dC_s^*-dC_s)dt\right].
	\end{align*}
	We may split the last conditional expectation into the following two parts:
	\begin{align*}
		& I^*=\E_\tau\left[\int_\tau^T \exp\left(\int_0^t \partial_uf(r,Y_r^*,U_r^*)dr\right)\partial_yf(t,Y_t^*,U_t^*)\int_\tau^t \theta_{t,s}dC_s^*dt\right],\\
		& I=\E_\tau\left[\int_\tau^T \exp\left(\int_0^t \partial_uf(r,Y_r^*,U_r^*)dr\right)\partial_yf(t,Y_t^*,U_t^*)\int_\tau^t \theta_{t,s}dC_sdt\right].
	\end{align*}
	By the Fubini theorem and Theorem (1.33) in Jacod \cite{J79}, we can deduce that
	\begin{align*}
		I^*=& \E_\tau\left[\int_\tau^T \int_s^T \exp\left(\int_0^t \partial_uf(r,Y_r^*,U_r^*)dr\right)\partial_yf(t,Y_t^*,U_t^*) \theta_{t,s}dtdC_s^*\right]\\
		=&\E_\tau\left[\int_\tau^T \E_s\left[\int_s^T \exp\left(\int_0^t \partial_uf(r,Y_r^*,U_r^*)dr\right)\partial_yf(t,Y_t^*,U_t^*) \theta_{t,s}dt\right]dC_s^*\right]\\
		=& \E_\tau\left[\int_\tau^T \nabla V(C^*)(s)dC^*_s\right].
	\end{align*}
	Applying Remark \ref{r3} implies that
	\begin{displaymath}
		I^*=M \E_\tau\left[\int_\tau^T M\psi_tdC^*_t\right].
	\end{displaymath}
	Similarly, we obtain that
		\begin{displaymath}
	I=\E_\tau\left[\int_\tau^T \nabla V(C^*)(s)dC_s\right]\leq M \E_\tau\left[\int_\tau^T M\psi_tdC_t\right].
	\end{displaymath}
	The above analysis yields that $\exp(\int_0^\tau \partial_uf(r,Y_r^*,U_r^*)dr)(U_\tau^*-U_\tau)\geq 0$.
\end{proof}
%\begin{lemma}\label{l5}
	
%\end{lemma}

\subsection{Construction of the optimal consumption plan}
  Theorem \ref{t1} confirms that the utility maximization problem \eqref{e1} has a unique solution, but it does not provide a method for constructing the optimal consumption plan. Inspired by \cite{BR} and the first-order conditions derived in Theorem \ref{t2}, we develop a forward-backward system (see Eqs. \eqref{e10}-\eqref{e9'} below) that characterizes a stochastic process known as the minimal level of satisfaction (see Definition 3.12 in \cite{BR}). The optimal consumption plan is the one that ensures the level of satisfaction remains above this minimal level in the most minimal way—specifically, the investor consumes only when their satisfaction level reaches the minimal threshold.  The minimal-level equation is an alternative to the Hamilton-Jacobi-Bellman equation within the dynamic programming approach; it is applicable   beyond the Markovian framework. To derive the minimal-level equation, we must first specify the dynamics for the level of satisfaction.

\begin{assumption}\label{a3}
	The   level of satisfaction $\eta$ and the consumption weights $\theta$ are given by
	\begin{displaymath}
		\eta_t=\eta e^{-\int_0^t \beta_sds},\ \theta_{t,s}=\beta_s e^{-\int_s^t \beta_rdr}, \ 0\leq s\leq t\leq T,
	\end{displaymath}
	where $\beta$ is a strictly positive, continuous function and $\eta\geq 0$ is a constant.
\end{assumption}

 For an adapted process $L$, we  define $Y^L$  by 
\begin{equation}\label{e10}
Y^L_t:=e^{-\int_0^t\beta_r dr}\left(\eta\vee \sup_{0\leq v\leq t}\left\{L_v e^{\int_0^v\beta_rdr}\right\}\right).
\end{equation}

%which is equivalent to
%\begin{equation}\label{e9}
%\begin{split}
%&\E_\tau\left[\int_\tau^T \exp\left(\int_\tau^t \partial_u f(r,Y_r^L,U_r^{L})dr\right)\partial_y f \left(t,\sup_{\tau\leq v\leq t}\left\{L_v e^{-\int_v^t \beta_sds}\right\}, U_t^{L}\right)\theta_{t,\tau}dt\right]\\
%=&M\psi_\tau\exp\left(\int_0^\tau \partial_u f(r,Y_r^L,U_r^{L})dr\right),
%\end{split}
%\end{equation}

By Lemma 3.9 in \cite{BR}, the process $C^L$ given by $C^L_{0-}=0$,
\begin{equation}\label{e11}
	C^L_t=\int_0^t Y^L_sds+\int_0^t \beta_s^{-1}dY^L_s,
\end{equation}
is nondecreasing and adapted; it is thus a consumption plan. Furthermore, $Y^L$ is the minimal level of satisfaction obtained by this consumption plan $C^L$, i.e.,
\begin{displaymath}
	Y^L_t=\min_{C\in\mathcal{X},Y^C\geq L} Y^C_t, \ t\in[0,T].
\end{displaymath}
Let  $U^L$ be the recursive utility obtained by the consumption plan $C^L$, i.e.,
\begin{equation}\label{UL}
U^L_s=\E_s\left[\int_s^T f(t,Y^{L}_t,U^L_t)dt\right].
\end{equation}
The above analysis indicates that the process $U^L$ given by \eqref{UL} can be regarded as optimal if the first-order condition
 
\begin{equation}\label{e9'}
	\E_\tau\left[\int_\tau^T \exp\left(\int_0^t \partial_u f(r,Y^L_r,U_r^{L})dr\right)\partial_y f\left(t,\sup_{\tau\leq v\leq t}\left\{L_v e^{-\int_v^t \beta_sds}\right\}, U_t^{L}\right)\theta_{t,\tau}dt\right]=M\psi_\tau,
\end{equation}
holds true for all stopping times $\tau$ and some suitable Lagrange multiplier $M$.
%for any $s\in[\tau,T]$,
%\begin{displaymath}
%	U_s^{\tau,L}=E_s[\int_s^T f(t,\sup_{\tau\leq v\leq t}\{L_v e^{-\int_v^t \beta_sds}\}, U_t^{\tau,L})dt].
%\end{displaymath}

\begin{theorem}
	Assume that forward-backward system  \eqref
{e10}, \eqref{UL}, and \eqref{e9'},  admits a unique solution $L$. %and the felicity function $f$ is submodular, i.e.,
	%\begin{displaymath}
	%	\frac{\partial^2 f(t,y,u)}{\partial y \partial u}\leq 0, \textrm{ for any } t\in[0,T].
	%\end{displaymath}
	Then, the consumption plan $C^L$ given by Eq. \eqref{e11} is optimal for the utility maximization problem \eqref{e1} with initial capital $w=\Psi(C^L)$.
\end{theorem}

\begin{proof}
	It is sufficient to show that $C^L$ satisfies the first-order conditions. By Eq. \eqref{e10}, it is easy to check that, for any $0\leq s\leq t\leq T$,
	\begin{displaymath}
		Y^L_t=\left\{Y^{L}_se^{-\int_s^t \beta_rdr}\right\}\vee \sup_{s\leq v\leq t}\left\{L_v e^{-\int_v^t \beta_rdr}\right\}.
	\end{displaymath}
	Since the utility function is concave, we may check that
	%	Noting that $Y^{C^L}_t\geq \sup_{\tau\leq v\leq t}\{L_v e^{-\int_v^t \beta_sds}\}$ for any $0\leq \tau\leq t\leq T$, by Proposition \ref{p1}, we have $U^{\tau,L}_t\leq U^L_s$, for any $\tau\leq s\leq T$. Therefore, we may check that for any $C\in\mathcal{X}$
		\begin{equation}\begin{split}\label{e12}
	&\E\left[\int_0^T \nabla V(C^L)(s)dC_s\right]\\
	=&\E\left[\int_0^T\left\{\int_s^T \exp\left(\int_0^t \partial_uf(r,Y_r^L,U_r^L)dr\right)\partial_y f\left(t, Y^L_t, U^L_t\right)\theta_{t,s}dt\right\}dC_s\right]\\
	\leq &\E\left[\int_0^T\left\{\int_s^T \exp\left(\int_0^t \partial_uf(r, Y_r^L, U^L_r)dr\right)\partial_y f\left(t, \sup_{s\leq v\leq t}\left\{L_v e^{-\int_v^t \beta_wdw}\right\}, U^L_t\right)\theta_{t,s}
	dt\right\}dC_s\right]\\
	=&\E\left[\int_0^T\E_s\left[\int_s^T \exp\left(\int_0^t \partial_uf(r, Y_r^L, U^L_r)dr\right)\partial_y f\left(t, \sup_{s\leq v\leq t}\left\{L_v e^{-\int_v^t \beta_wdw}\right\}, U^L_t\right)\theta_{t,s}
	dt\right]dC_s\right]\\
	=&\E\left[\int_0^T M\psi_sdC_s\right].	
	\end{split}\end{equation}
	%\begin{equation}\begin{split}\label{12}
	%	&E[\int_0^T \nabla V(C^L)(s)dC_s]\\
	%	=&E[\int_0^T\{\int_s^T \partial_y f(t, \{Y^{L}_se^{-\int_s^t \beta_rdr}\}\vee \sup_{s\leq v\leq t}\{L_v e^{-\int_v^t \beta_wdw}\}, U^L_t)\theta_{t,s}\\
	%	&\times \exp(\int_s^t \partial_uf(r,\{Y^{L}_se^{-\int_s^t \beta_wdw}\}\vee \sup_{s\leq v\leq r}\{L_v e^{-\int_v^t \beta_wdw}\}, U^L_r)dr+\int_0^s \partial_uf(r,Y_r^L,U_r^L)dr)dt\}dC_s]\\
	%	\leq &E[\int_0^T\{\int_s^T \exp(\int_0^t \partial_uf(r, \sup_{s\leq v\leq r}\{L_v e^{-\int_v^t \beta_wdw}\}, U^L_r)dr)\partial_y f(t, \sup_{s\leq v\leq t}\{L_v e^{-\int_v^t \beta_wdw}\}, U^L_t)\theta_{t,s}
	%	dt\}dC_s]\\
	%	=&E[\int_0^TE_s[\int_s^T \exp(\int_0^t \partial_uf(r, \sup_{s\leq v\leq r}\{L_v e^{-\int_v^t \beta_wdw}\}, U^L_r)dr)\partial_y f(t, \sup_{s\leq v\leq t}\{L_v e^{-\int_v^t \beta_wdw}\}, U^L_t)\theta_{t,s}
	%	dt]dC_s]\\
	%	=&E[\int_0^T M\psi_sdC_s].	
	%\end{split}\end{equation}
	Since the above inequality holds for any $C\in\mathcal{X}$, by Meyer's optional section theorem, we derive that $\nabla V(C^L)\leq M\psi$. By Lemma 3.9 in \cite{BR}, for each fixed $\omega$, $t\in[0,T]$ is a point of increase of $C^L_{\cdot}(\omega)$ (which means $dC^L_{t}(\omega)=C^L_{t}(\omega)-C^L_{t-}(\omega)>0$) only when $Y^L_t(\omega)=L_t(\omega)$. Therefore, if we replace $C$ in Eq. \eqref{e12} by $C^L$, we will end up with an equation, which is the desired flat off condition. 
\end{proof}

\begin{remark}
Compared to the backward equation that characterizes the minimal level of satisfaction in the time-additive case (see Eq. (17) in \cite{BR}), our equation \eqref{e9'} depends on the entire path of the process $L$. Consequently, unlike the time-additive case, Theorem 3 in \cite{BE} does not guarantee the existence of a solution to the forward-backward system \eqref{e10}-\eqref{e9'}. 
\end{remark}

\subsection{Quantitative properties of the optimal consumption}

 In this subsection, we study some quantitative properties of the optimal consumption plan. For this purpose, we assume that the interest rate $r$ and the depreciation rate $\beta$ are  constants. Hence, the level of satisfaction is given by 
\begin{equation}\label{YCt}
Y^C_t=\eta e^{-\beta t}+\int_0^t \beta e^{-\beta(t-s)}dC_s.
\end{equation}
%with constants $\eta,\beta>0$. That is, $\theta_{s,t}=\beta e^{-\beta(s-t)}$. 
We define
\begin{align}\label{1}
F_s(C):=\exp\left(\int_0^s \partial_u f(r,Y_r^C,U_r^C)dr\right)\partial_y f(s,Y_s^C,U_s^C), \ s\in[0,T].
\end{align}

Let $\tau_0<\tau_1$ be two stopping times. Similar with the free interval for the irreversible investment problem introduced in \cite{RS},  $[\tau_0,\tau_1]$ is called a free interval if the consumption occurs at a strictly positive rate throughout the interval.

\begin{theorem}\label{freeinterval}
Suppose that $[\tau_0,\tau_1]$ is a free interval. Then, for all $t\in(\tau_0,\tau_1)$, we have
\begin{equation}\label{8}
F_t(C)=M\frac{r+\beta}{\beta}\psi_t,
\end{equation}
where $M$ is the Lagrange multiplier in Theorem \ref{t2}.
\end{theorem}

\begin{proof}
It follows from the first-order conditions in Theorem \ref{t2} that, for any $t\in(\tau_0,\tau_1)$,
\begin{align}\label{2}
e^{\beta t}\E_t\left[\int_t^T F_s(C)\beta e^{-\beta s}ds\right]=M\psi_t=M e^{-rt}\frac{d\P^*}{d\P_0}\Big|_{\mathcal{F}_t}.
\end{align} 
Note that 
\begin{equation}\label{3}\begin{split}
&\E_t\left[\frac{d\P^*}{d\P_0}\int_t^T F_s(C)\beta e^{-\beta s}\left(\frac{d\P^*}{d\P_0}\Big|_{\mathcal{F}_s}\right)^{-1}ds\right]\\
=&\E_t\left[\int_t^T \frac{d\P^*}{d\P_0}F_s(C)\beta e^{-\beta s}\left(\frac{d\P^*}{d\P_0}\Big|_{\mathcal{F}_s}\right)^{-1}ds\right]\\
=&\E_t\left[\int_t^T \frac{d\P^*}{d\P_0}\Big|_{\mathcal{F}_s}F_s(C)\beta e^{-\beta s}\left(\frac{d\P^*}{d\P_0}\Big|_{\mathcal{F}_s}\right)^{-1}ds\right]\\
=&\E_t\left[\int_t^T F_s(C)\beta e^{-\beta s}ds\right].
\end{split}\end{equation}
Combining \eqref{2}, \eqref{3} and applying the Bayesian rule, for any $t\in(\tau_0,\tau_1)$, we have
\begin{equation}
\begin{split}\label{4}
&\E^*_t\left[\int_t^T F_s(C)\beta e^{-\beta s}\left(\frac{d\P^*}{d\P_0}\Big|_{\mathcal{F}_s}\right)^{-1}ds\right]\\=&\E_t\left[\frac{d\P^*}{d\P_0}\int_t^T F_s(C)\beta e^{-\beta s}\left(\frac{d\P^*}{d\P_0}\Big|_{\mathcal{F}_s}\right)^{-1}ds\right]\left(\frac{d\P^*}{d\P_0}\Big|_{\mathcal{F}_t}\right)^{-1}\\
=&\E_t\left[\int_t^T F_s(C)\beta e^{-\beta s}ds\right]\left(\frac{d\P^*}{d\P_0}\Big|_{\mathcal{F}_t}\right)^{-1}=M e^{-(r+\beta)t}.
\end{split}\end{equation}
It follows that 
\begin{equation}\label{5}
\E^*_t\left[\int_0^T F_s(C)\beta e^{-\beta s}\left(\frac{d\P^*}{d\P_0}\Big|_{\mathcal{F}_s}\right)^{-1}ds\right]=M e^{-(
r+\beta)t}+\int_0^t F_s(C)\beta e^{-\beta s}\left(\frac{d\P^*}{d\P_0}\Big|_{\mathcal{F}_s}\right)^{-1}ds.
\end{equation}
Since the left-hand side of \eqref{5}, denoted by $M^*_t$, is an $\P^*$-martingale while the right-hand side is absolutely continuous,  $M^*_t$ should be a constant. 
%The left-hand side is a martingale under $\P^*$. Let $B^*$ be the Brownian motion under $\P^*$, we have
%\begin{align}\label{6}
%\E^*_t\left[\int_0^T F_s(C)\beta e^{-\beta s}ds\frac{d\P_0}{d\P^*}\Big|_{\mathcal{F}_s}ds\right]=\E^*\left[\int_0^T F_s(C)\beta e^{-\beta s}\frac{d\P_0}{d\P^*}\Big|_{\mathcal{F}_s}ds\right]+\int_0^t Z^*_s dB^*_s.
%\end{align}
%Then, we have $Z^*\equiv 0$. That is, for any $t\in(\tau_0,\tau_1)$,
%\begin{align}\label{7}
%M e^{-(r+\beta)t}+\int_0^t F_s(C)\beta e^{-\beta s}\frac{d\P_0}{d\P^*}\Big|_{\mathcal{F}_s}ds=\E^*\left[\int_0^T F_s(C)\beta e^{-\beta s}\frac{d\P_0}{d\P^*}\Big|_{\mathcal{F}_s}ds\right].
%\end{align}
Taking derivatives on both sides of \eqref{5}, we obtain that 
\begin{align*}
-M(r+\beta)e^{-(r+\beta)t}+\beta e^{-\beta t} F_t(C)\left(\frac{d\P^*}{d\P_0}\Big|_{\mathcal{F}_t}\right)^{-1}=0,
\end{align*}
which implies the desired result.
\end{proof}

The following theorem investigates the case when there is a lump-sum consumption at some stopping time $\tau$.
\begin{theorem}\label{lumpsum}
Suppose that the optimal consumption plan has a jump at a stopping time $\tau$ and the state-price density $\psi$ has right-continuous sample paths. Then, we have 
\begin{align*}
F_\tau(C)\geq M\frac{r+\beta}{\beta}\psi_\tau, \textrm{ a.s. on } \{\tau<T\},
\end{align*}
where $M$ is the Lagrange multiplier in Theorem \ref{t2}.
\end{theorem}

\begin{proof}
Let $\tau$ be a stopping time with $\Delta C_\tau>0$ a.s. on $\tau<T$. From now on, we work on the set $\{\tau<T\}$. Fix $\varepsilon>0$. Let $\sigma:=\inf\{t\geq \tau: F_t(C)-M\frac{r+\beta}{\beta}\psi_t\geq -\varepsilon\}$. Similar as \eqref{4}, the first-order conditions at time $\tau$ and $\sigma$ imply that 
\begin{align*}
&\E^*_\tau\left[\int_\tau^T F_s(C)\beta e^{-\beta s}\left(\frac{d\P^*}{d\P_0}\Big|_{\mathcal{F}_s}\right)^{-1}ds\right]=M e^{-(r+\beta)\tau},\\
&\E^*_\sigma\left[\int_\sigma^T F_s(C)\beta e^{-\beta s}\left(\frac{d\P^*}{d\P_0}\Big|_{\mathcal{F}_s}\right)^{-1}ds\right]\leq M e^{-(r+\beta)\sigma},
\end{align*}
which are equivalent to 
\begin{align*}
&\E^*_\tau\left[\int_\tau^T \left(F_s(C)\beta e^{-\beta s}\left(\frac{d\P^*}{d\P_0}\Big|_{\mathcal{F}_s}\right)^{-1}-M(r+\beta)e^{-(r+\beta )s}\right)ds\right]=M e^{-(r+\beta)T},\\
&\E^*_\sigma\left[\int_\sigma^T \left(F_s(C)\beta e^{-\beta s}\left(\frac{d\P^*}{d\P_0}\Big|_{\mathcal{F}_s}\right)^{-1}-M(r+\beta)e^{-(r+\beta )s}\right)ds\right]\leq M e^{-(r+\beta)T}.
\end{align*}
We obtain by taking conditional expectation at time $\tau$ of their differences
\begin{align*}
0 & \leq \E^*_\tau\left[\int_\tau^\sigma\left(F_s(C)\beta e^{-\beta s}\left(\frac{d\P^*}{d\P_0}\Big|_{\mathcal{F}_s}\right)^{-1}-M(r+\beta)e^{-(r+\beta )s}\right)ds\right]\\
& \leq -\varepsilon\E^*_\tau\left[\int_\tau^\sigma \beta e^{-\beta s}\left(\frac{d\P^*}{d\P_0}\Big|_{\mathcal{F}_s}\right)^{-1}ds\right],
\end{align*}
where we have used the definition of $\sigma$ in the last inequality. Therefore, we have $\sigma=\tau$, i.e., for any $\varepsilon$, $F_\tau(C)-M\frac{r+\beta}{\beta}\psi_\tau\geq -\varepsilon$. Since $\varepsilon>0$ can be chosen arbitrarily, we obtain that 
\begin{align*}
F_\tau(C)\geq M\frac{r+\beta}{\beta}\psi_\tau, \textrm{ a.s. on } \{\tau<T\}.
\end{align*}
\hfill
\end{proof}

Motivated by the irreversible investment model studied by \cite{RS}, we introduce the notion of surprise for a deterministic time $t$.
We say that the model has a surprise at a deterministic time $t\in[0,T]$ if either $\mathcal{F}_{t-}\neq \mathcal{F}_t$ or $\P(\Delta \psi_t\neq 0)>0$ (i.e., $\P\left(\Delta \frac{d\P^*}{d\P_0}\Big|_{\mathcal{F}_t}\neq 0\right)>0$).

The following theorem states that lump-sum consumption occurs only in response to surprises.

\begin{theorem}\label{nosurprise}
Suppose that  
$$
\E^*\left[\sup_{t\in[0,T]}\left\{\partial_y f(t,Y^C_t, U^C_t)\left(\frac{d\P^*}{d\P_0}\Big|_{\mathcal{F}_t}\right)^{-1}\right\}\right]<\infty.
$$
If the model has no surprise at time $t\in[0,T]$, then there is no lump-sum consumption at $t$, i.e., $\P(\Delta C_t\neq 0)=0$. 
\end{theorem}

\begin{proof}
Suppose that $\P(A)>0$, where $A=\{\Delta C_t>0\}$. Since the first-order conditions are binding on $A$, recalling \eqref{4}, we have
\begin{align*}
\E^*_t\left[\int_t^T F_s(C)\beta e^{-\beta s}\left(\frac{d\P^*}{d\P_0}\Big|_{\mathcal{F}_s}\right)^{-1}ds\right]=M e^{-(r+\beta)t}, \textrm{ on } A.
\end{align*} 
Let $\{t_n\}_{n\in\mathbb{N}}$ be such that $t_n\uparrow t$. It follows again from the first-order conditions that 
\begin{align*}
\E^*_{t_n}\left[\int_{t_n}^T F_s(C)\beta e^{-\beta s}\left(\frac{d\P^*}{d\P_0}\Big|_{\mathcal{F}_s}\right)^{-1}ds\right]\leq M e^{-(r+\beta)t_n}.
\end{align*}
By taking differences and the expectation, we get (similar as the third equation in the proof of Theorem 6.5 in \cite{RS})
\begin{align*}
\E^*\left[I_A\E^*_{t_n}\left[\int_{t_n}^t F_s(C)\beta e^{-\beta s}\left(\frac{d\P^*}{d\P_0}\Big|_{\mathcal{F}_s}\right)^{-1}ds\right]\right]\leq M\E^*\left[I_A(e^{-(r+\beta)t_n}-e^{-(r+\beta)t})\right].
\end{align*}
Therefore, we have
\begin{align}\label{H1}
\E^*\left[I_A\E^*_{t_n}\left[\int_{t_n}^t \left(F_s(C)\beta e^{-\beta s}\left(\frac{d\P^*}{d\P_0}\Big|_{\mathcal{F}_s}\right)^{-1}-M(r+\beta)e^{-(r+\beta)s}\right)ds\right]\right]\leq 0.
\end{align}
Noting that at time $t$, we have $\lim_{s\uparrow t} \left(\frac{d\P^*}{d\P_0}\Big|_{\mathcal{F}_s}\right)^{-1}=\left(\frac{d\P^*}{d\P_0}\Big|_{\mathcal{F}_t}\right)^{-1}$ and $\lim_{s\uparrow t}C_s=C_{t-}$, a.s. It follows that $\lim_{s\uparrow t}Y^C_s=Y^C_{t-}$ and $\lim_{n\rightarrow \infty}Z_n=Z$, a.s., where
\begin{align*}
&Z_n:=\frac{1}{t-t_n}\int_{t_n}^t \left(F_s(C)\beta e^{-\beta s}\left(\frac{d\P^*}{d\P_0}\Big|_{\mathcal{F}_s}\right)^{-1}-M(r+\beta)e^{-(r+\beta)s}\right)ds,\\
&Z:=F_{t-}(C)\beta e^{-\beta t}\left(\frac{d\P^*}{d\P_0}\Big|_{\mathcal{F}_t}\right)^{-1}-M(r+\beta)e^{-(r+\beta)t},\\
&F_{t-}(C)=\exp\left(\int_0^t \partial_u f(r,Y_r^C,U_r^C)dr\right)\partial_y f(t,Y_{t-}^C,U_{t-}^C).
\end{align*}
It is easy to check that 
$$-M(r+\beta)\leq Z_n\leq \beta e^{KT}\sup_{t\in[0,T]}\left\{\partial_y f(t,Y^C_t, U^C_t)\left(\frac{d\P^*}{d\P_0}\Big|_{\mathcal{F}_t}\right)^{-1}\right\}=:\bar{Z},$$
where $K$ is the Lipschitz constant of $f$ w.r.t $u$. Since $\bar{Z}$ is $\P^*$-integrable, the convergence $Z_n\rightarrow Z$ also holds in $L^1(\P^*)$. By a martingale convergence argument and noting that $\mathcal{F}_{t-}=\mathcal{F}_t$, we have
\begin{align}\label{H2}
\lim\E^*_{t_n}[Z_{n}]=\E^*_{t-}[Z]=Z \textrm{ a.s. and in } L^1(\P^*).
\end{align}
Set $\tau:=t I_A+T I_{A^c}$. Since $Y^C_t> Y^C_{t-}$ on $A$, by Theorem \ref{lumpsum} and recalling that $f$ is strictly concave, we have
\begin{align*}
F_{t-}(C)>F_t(C)\geq M\frac{r+\beta}{\beta}\psi_t, \textrm{ on } A,
\end{align*}
which is equivalent to $Z>0$ on $A$. By Fatou's lemma and \eqref{H1}, \eqref{H2}, we finally obtain that 
\begin{align*}
0&<\E^*[I_A Z]=\E^*[I_A \lim_{n\rightarrow \infty}Z_n]\leq \liminf_{n\rightarrow \infty}\E^*[I_A Z_n]\\
&=\liminf_{n\rightarrow \infty}\E^*[I_A (Z_n-\E^*_{t_n}[Z_n])]+\liminf_{n\rightarrow \infty}\E^*[I_A \E^*_{t_n}[Z_n]]\leq 0,
\end{align*}
which is a contradiction.
\end{proof}

%Recall that in the time-additive case studied by \cite{BR}, if the state price density is characterized by a geometric Brownian motion, the optimal consumption behaves in a singular way. That is, there is no free interval.

\section{A Case Study}

%\subsubsection{Pure Poisson Setting}
In this section, we are going to provide a more explicit structure of optimal consumption plans. For this purpose, we suppose that the level of satisfaction is given by \eqref{YCt} and the state price density is characterized by the  geometric Poisson process 
\begin{align*}
    \psi_t=\exp\left\{\theta N_t-(r+\lambda(e^\theta-1))t\right\},
\end{align*}
where $\theta>0$ and $N$ is a Poisson process with intensity $\lambda>0$ under the probability $\P_0$. The following assumption for the felicity function $f$ and the parameters $r,\beta,\lambda,\theta$ will play a crucial role in the analysis.

\begin{assumption}\label{a4}
    Suppose that for any $(t,y,u)\in[0,T]\times \mathbb{R}_+\times \mathbb{R}$, we have \begin{equation}\label{parameterconstraints}
     \mathfrak{L}f(t,y,u):=(r+\lambda(e^\theta-1))\partial_y f-\beta y \partial_y^2f+\partial_t\partial_y f+\partial_u f\partial_y f-f\partial_u\partial_y f> 0.  
    \end{equation}
\end{assumption}

Let us give an economic interpretation of the assumption. 
Suppose that the aggregator $f$ does not depend on $t$. Define
\begin{align*}
    \rho^f:=\frac{f\partial_u \partial_y f}{\partial_y f}-\partial_u f, \ \epsilon^f:=-\frac{\partial_y f}{y\partial_y^2 f}.
\end{align*}
$\epsilon^f$ denotes the intertemporal elasticity of substitution, while $\rho^f$, as demonstrated in \cite{E87}, represents the endogenous discount rate. The condition $\mathfrak{L}f(y,u)\leq 0$ is equivalent to 
\begin{align*}
    \rho^f\geq r+\lambda(e^\theta-1)+\frac{\beta}{\epsilon^f},
\end{align*}
which amounts to say that the endogenous discount rate is no less than the sum of the interest rate, the market risk premium and the ratio between the depreciation rate and the intertemporal elasticity of substitution. Now consider  the important special case of the  Epstein-Zin felicity function
\begin{equation}\label{barf}
    {f}(y,u)=\frac{\delta}{1-\frac{1}{{\alpha}}}y^{1-\frac{1}{{\alpha}}}[(1-{\rho})u]^{1-\frac{1}{\sigma}}-\delta\sigma u,
\end{equation}
where $\sigma=(1-{\rho})(1-\frac{1}{{\alpha}})^{-1}$. Here, $\rho$ is the agent's relative risk aversion. Besides, we have $\epsilon^f=\alpha$ and $\rho^f=\delta$ (see Lemma 4.6 in \cite{LRY}).  Condition $\mathfrak{L}f(y,u)\leq 0$ simplifies to 
$$\delta \geq \lambda(e^\theta-1)+r+\frac{\beta}{\alpha}.$$
%Proposition \ref{lumpconsumptionat0} indicates that if the agent is quite impatient (i.e., the discount rate is high), he optimally consumes all his wealth at time $0$.

We first consider  the case when Assumption \ref{a4} does not hold. In that case,  it is optimal  to consume all wealth $w$ at time $0$.
\begin{proposition}\label{lumpconsumptionat0}
    Suppose that for any $(t,y,u)\in[0,T]\times \mathbb{R}_+\times \mathbb{R}$, we have 
    \begin{align*}
     \mathfrak{L}f(t,y,u)\leq 0.  
    \end{align*}
    Then, the optimal consumption plan $C$ is given by $C_t=w$, $t\in[0,T]$. 
   % \fr{Compare to Remark 4.4 in the JME paper. Agent is very impatient relative to interest rate, depreciation and ...}
\end{proposition}

\begin{proof}%[Proof of Proposition \ref{lumpconsumptionat0}]
Let $Y^C$ and $U^C$ be the level of satisfaction and utility induced by $C$, respectively. In this case, we have
\begin{align*}
    Y^C_t=(\beta w+\eta)e^{-\beta t}.
\end{align*}
Clearly, $Y^C$ is deterministic, which implies that the utility $U^C$ is also deterministic. Therefore, we have 
\begin{equation*}\label{dUtdetermine}
        dU^C_t=-f(t,Y^C_t,U^C_t)dt.
\end{equation*}
Set $\widetilde{F}_t(C)=\nabla V(C)(t)\exp((r+\lambda(e^\theta-1))t)$. We claim that $\{\widetilde{F}_t(C)\}_{t\in[0,T]}$ is nonincreasing in $t$. For simplicity, we omit the brackets in $\widetilde{F}$ and $F$. By a simple calculation and recalling the definition for $\nabla V(C)(t)$ and $F_t(C)$ (see Eqs. \eqref{nablaV} and \eqref{1}), we have 
\begin{equation}\label{widetildeF}\begin{split}
&\widetilde{F}'_t=\beta(r+\lambda(e^\theta-1)+\beta)\widetilde{F}_t-\beta X_tF_t,\\
&\widetilde{F}''_t=\beta(r+\lambda(e^\theta-1)+\beta)\widetilde{F}'_t-\beta X_t F_t^1\mathfrak{L}f(t,Y^C_t,U^C_t),
\end{split}\end{equation}
where
\begin{align*}
   X_t=\exp\left((r+\lambda(e^\theta-1))t\right), \  F^1_t(C)=\exp\left(\int_0^t \partial_u f(r,Y_r^C,U_r^C)dr\right).
\end{align*}
Set $a=\beta(r+\lambda(e^\theta-1)+\beta)$ and $
    b_t=\beta X_t F_t^1\mathfrak{L}f(t,Y^C_t,U^C_t)$. By the second equation in \eqref{widetildeF}, we have
\begin{align*}
    e^{-at}\widetilde{F}'_t=e^{-aT}\widetilde{F}'_T+\int_t^T e^{-as}b_sds.
\end{align*}
Noting that $\widetilde{F}'_T=-\beta X_TF_T \leq 0$ and $b_t\leq 0$, it follows that $\widetilde{F}'_t\leq 0$. Therefore, $\{\widetilde{F}_t(C)\}_{t\in[0,T]}$ is nonincreasing in $t$.

Now, we define 
\begin{align*}
    M:=\widetilde{F}_0(C).
\end{align*}
It is easy to check that $\nabla V(C)(0)=M\psi_0$ and for any $t\in(0,T]$, we have
\begin{align*}
    \nabla V(C)(t)\leq M X_t^{-1}\leq M\psi_t.
\end{align*}
All the above analysis indicates that the first-order conditions listed in Theorem \ref{t2} are satisfied for $C\equiv w$ with Lagrange multiplier $\widetilde{F}_0(C)$. Therefore, the optimal plan is to consume all wealth $w$ at time $0$.
\end{proof}

In the following, suppose that Assumption \ref{a4} holds. Let $\sigma_i$ be the point of jump for $N$, $i=1,2,\cdots$. By Theorem \ref{nosurprise}, the free interval $[\tau_0,\tau_1]$ must be contained in some $(\sigma_i,\sigma_{i+1})$, $i=1,2,\cdots$. For any $t\in (\tau_0,\tau_1)\subset (\sigma_i,\sigma_{i+1})$, we have
\begin{align*}
    \psi_t=\exp\left\{\theta N_{\sigma_i}-(r+\lambda(e^\theta-1))t\right\}.
\end{align*}
Therefore, for $t\in(\tau_0,\tau_1)$ we obtain that 
\begin{align*}
    d\psi_t=-(r+\lambda(e^\theta-1))\exp\left\{\theta N_{\sigma_i}-(r+\lambda(e^\theta-1))t\right\}dt=-(r+\lambda(e^\theta-1))\psi_tdt.
\end{align*}
%We define
%\begin{align*}
 %   F^1_t(C):=\exp\left(\int_0^t \partial_u f(r,Y_r^C,U_r^C)dr\right),\ F^2_t(C):=\partial_y f(t,Y_t^C,U_t^C).
%\end{align*}
It is easy to check that 
\begin{equation*}\label{dFt}\begin{split}
    dF_t(C)
    =&F_t(C)\partial_u f(t,Y_t^C,U_t^C)dt+F^1_t(C)\Big[\partial_t\partial_y f(t,Y^C_t,U^C_t)dt\\
    &+\partial_y^2 f(t,Y^C_t,U^C_t)dY^C_t+\partial_u\partial_y f(t,Y^C_t,U^C_t)dU^C_t\Big].
\end{split}\end{equation*}
Taking derivatives on both sides of \eqref{8}, for $t\in(\tau_0,\tau_1)$, we have 
\begin{equation}\label{dFt''}\begin{split}
    -M\frac{r+\beta}{\beta}(r+\lambda(e^\theta-1))\psi_tdt
    =&F_t(C)\partial_u f(t,Y_t^C,U_t^C)dt+F^1_t(C)[\partial_t\partial_y f(t,Y^C_t,U^C_t)dt\\
    &+\partial_y^2 f(t,Y^C_t,U^C_t)dY^C_t+\partial_u\partial_y f(t,Y^C_t,U^C_t)dU^C_t].
\end{split}\end{equation}
Recalling the dynamics for $U^C$ in Remark \ref{BSDEforU} and noting that $\tilde{v}(dt,dx)=\delta_1(x)(dN_t-\lambda dt)$, we have%Noting that the filtration is generated by a Poisson process, by the martingale representation theorem, we have
\begin{equation}\label{dUt}
        dU^C_t=-f(t,Y^C_t,U^C_t)dt+\bar{Z}_tdB_t+\bar{\Psi}_t(1)(dN_t-\lambda dt).
\end{equation}
For shorter notation, we write $\bar{\Psi}_t=\bar{\Psi}_t(1)$. Let $c_t$ be the consumption rate. For $t\in(\tau_0,\tau_1)$, we have  \begin{equation}\label{NtYt}
 dN_t=0 \textrm{ and }   dY^C_t=\beta(dC_t-Y_tdt)=\beta(c_t-Y_t)dt.
\end{equation} 
%Taking derivatives on both sides of \eqref{dUt} yields that 
%\begin{align*}
%    dU^C_t=-f(t,Y^C_t,U^C_t)dt-\lambda Z_t dt.
%\end{align*} 
Plugging Eqs. \eqref{dUt} and \eqref{NtYt} into \eqref{dFt''} and using the fact that $F_t(C)=M(r+\beta)\beta^{-1}\psi_t$, we obtain that $\bar{Z}_t=0$ and 
%\begin{equation}\label{dYt'}
 %   \frac{dY^C_t}{dt}=-\frac{\partial_u f\partial_y f+\partial_t\partial_y f-(f+\lambda \bar{\Psi}_t\partial_u\partial_y f+(r+\lambda(e^\theta-1))\partial_y f}{\partial_y^2 f}.
%\end{equation}
%Recalling that $dC_t=\beta^{-1}dY_t+Y_tdt$, we have
\begin{equation}\label{consumptionrate}\begin{split}
    c_t=&-\frac{\partial_u f\partial_y f+\partial_t\partial_y f-(f+\lambda \bar{\Psi}_t)\partial_u\partial_y f+(r+\lambda(e^\theta-1))\partial_y f}{\beta\partial_y^2 f}+Y_t\\
    %=&\frac{(r+\lambda(e^\theta-1))\epsilon^f}{\beta}Y_t+Y_t-\frac{\epsilon^f\rho^f}{\beta}Y_t+\frac{\lambda \bar{\Psi}_t\partial_u\partial_y f}{\beta\partial_y^2 f}\\
    =&\frac{1}{\beta} \left[ \left( r - \rho^f\right) \epsilon^f + \beta\right] Y_t+ \frac{\lambda(e^\theta-1) \epsilon^f}{\beta} Y_t + \frac{\lambda \bar{\Psi}_t\partial_u\partial_y f}{\beta\partial_y^2 f}.
\end{split}\end{equation}
The first term in the right-hand side of \eqref{consumptionrate} coincides with the consumption rate in the deterministic setting (see Eq. (10) in \cite{LRY}). The second term is determined by the market risk premium multiplied by the intertemporal elasticity of substitution and divided by the depreciation rate. The last term is due to stochastic differential utility.

In the following, we would like to find a more explicit expression for $\bar{\Psi}$ in \eqref{consumptionrate}. To this end, let us consider the case where the felicity function is of  Epstein-Zin's type given in \eqref{barf}. 
Let $g:\mathbb{R}_+\times\mathbb{R}\rightarrow \mathbb{R}$ be Borel measurable and $M:\mathbb{R}\times\mathbb{R}\rightarrow \mathbb{R}$ be such that $M(v,\cdot)$ is of class $C^2$ for any $v\in\mathbb{R}$.  Recalling Definition 6.1 in \cite{KS}, the stochastic differential utility $\mathfrak{u}(C)$ associated to $(g,M)$  is defined by $V^C_0$, where
\begin{equation*}
    V_t^C=\E_t\left[\int_t^T \left(g(Y^C_s,V_s^C)+\frac{1}{2}A(V_s^C)Z_s^2+\int_{\mathbb{R}_*}J(V_s^C,\Psi_s(x))\theta(dx)\right)ds\right]
\end{equation*}
and $dV^C_t$ satisfies
\begin{equation*}
    dV^C_t=-\left(g(Y^C_t,V_t^C)+\frac{1}{2}A(V_t^C)Z_t^2+\int_{\mathbb{R}_*}J(V_t^C,\Psi_t(x))\theta(dx)\right)dt+Z_tdB_t+\int_{\mathbb{R}_*}\Psi_t(x)\tilde{v}(dt,dx).
\end{equation*}
Here, the variance multiplier $A$ and the jump term $J$ associated to $\mathfrak{m}$ (certainty equivalent) are defined by 
\begin{equation*}
    A(v):=\partial_w^2 M(v,w)|_{w=v} \textrm{ and } J(v,\psi):=M(v,v+\psi)-M(v,v)-\psi.
\end{equation*}
Intuitively, $A$ represents the investor's aversion towards diffusion risk, $J$ captures aversion towards jump risk. In the following, we always omit the superscript $C$.

Let %$\Phi(x)=\frac{1}{\alpha}x^\alpha$, $g(c,v)=\frac{\beta}{\rho}\frac{c^\rho-v^\rho}{v^{\rho-1}}$, $A(v)=\frac{\alpha-1}{v}$ and $M(v,w)=\frac{1}{\alpha}w^\alpha v^{1-\alpha}$.
\begin{align*}
 \Phi(x)=\frac{1}{1-\rho}x^{1-\rho}, \ g(y,v)=\frac{\delta}{1-\frac{1}{\alpha}}\left(y^{1-\frac{1}{\alpha}}v^{\frac{1}{\alpha}}-v\right), \ A(v)=-\frac{\rho}{v}, \ M(v,w)=\frac{1}{1-\rho}w^{1-\rho}v^{\rho}.   
\end{align*}
It is easy to check that $A(v)=\frac{\Phi''(v)}{\Phi'(v)}$ and $M(v,w)=\frac{\Phi(w)}{\Phi'(v)}$, which implies that $A,M$ satisfy the conditions in Corollary 7.3 and Theorem 7.4 in \cite{KS}. Let $U_t=\Phi(V_t)$. Applying It\^{o}'s formula, we obtain that 
\begin{align*}
    dU_t=&-\left(\Phi'(V_t)g(Y_t,V_t)+\frac{1}{2}\Phi'(V_t)A(V_t)Z_t^2+\Phi'(V_t)\int_{\mathbb{R}_*}J(V_t,\Psi_t(x))\theta(dx)-\frac{1}{2}\Phi''(V_t)Z_t^2\right)dt\\
    &+\int_{\mathbb{R}_*}(\Phi(V_t+\Psi_t(x))-\Phi(V_t)-\Psi_t\Phi'(V_t))\theta(dx)dt+\Phi'(V_t)Z_tdB_t\\
    &+\int_{\mathbb{R}_*}(\Phi(V_{t-}+\Psi_t)-\Phi(V_{t-}))\tilde{v}(dt,dx)\\
    =&-\frac{\delta}{1-\frac{1}{\alpha}}\left(Y_t^{1-\frac{1}{\alpha}}V_t^{\frac{1}{\alpha}-\rho}-V_t^{1-\rho}\right)dt+V_t^{-\rho}Z_t dB_t\\ &+\int_{\mathbb{R}_*}\frac{1}{1-\rho}((V_{t-}+\Psi_t)^{1-\rho}-\Phi(V_{t-})^{1-\rho})\tilde{v}(dt,dx)\\
    =&-{f}(Y_t,U_t)dt+\bar{Z}_t dB_t+\int_{\mathbb{R}_*}\bar{\Psi}_t\tilde{v}(dt,dx),
\end{align*}
where $\bar{Z}_t=V_t^{-\rho}Z_t$, $\bar{\Psi}_t=\frac{1}{1-\rho}((V_{t-}+\Psi_t)^{1-\rho}-V_{t-}^{1-\rho})$. Therefore, $U$ is exactly the utility induced by the level of satisfaction $Y$ with felicity function $f$ given by \eqref{barf}. Moreover, since for $t\in(\tau_0,\tau_1)$, the Poisson process $N$ has no jump.  $\bar{\Psi}$ in \eqref{consumptionrate} can be represented as 
\begin{align*}
    \bar{\Psi}_t=\frac{1}{1-\rho}((V_{t}+\Psi_t)^{1-\rho}-V_{t}^{1-\rho})=(J(V_{t},\Psi_t)+\Psi_t)V_{t}^{-\rho}.
\end{align*}
% and 
%\begin{align*}
 %   {f}(c,v)=\frac{\beta}{\rho}\frac{c^\rho-(\alpha v)^{\frac{\rho}{\alpha}}}{(\alpha v)^{\frac{\rho}{\alpha}-1}}.
%\end{align*}
% The the above $f$ is exactly the felicity function in \eqref{barf}.  Therefore, we could get that 
%\begin{equation}
 %   U_t=\E_t\left[\int_t^T{f}(Y_s,U_s)ds\right]=\E_t\left[\int_t^T\frac{\beta}{\rho}\frac{Y_s^\rho-(\alpha U_s)^{\frac{\rho}{\alpha}}}{(\alpha U_s)^{\frac{\rho}{\alpha}-1}}ds\right].
%\end{equation}

%From my point of view, $\bar{\Psi}_t$ given above corresponds to $Z_t$ in \eqref{dYt'} and \eqref{dCt} and $\bar{V}_t$ corresponds to $U_t$ in \eqref{dYt'} and \eqref{dCt}. 
Using the first expression of $\bar{\Psi}$, we have
\begin{equation}\label{zt1}
    \frac{ \bar{\Psi}_t\partial_u\partial_y f}{ \partial_y^2 f}=-\frac{1-{\alpha}{\rho}}{1-{\rho}}\frac{(V_{t}+\Psi_t)^{1-{\rho}}-V_{t}^{1-{\rho}}}{V_{t}^{1-{\rho}}}Y_t=-\frac{1-{\alpha}{\rho}}{1-{\rho}}\left(\left(1+\frac{\Psi_t}{V_{t}}\right)^{1-{\rho}}-1\right)Y_t.
\end{equation}
%Another expression is as follows. Note that 
%\begin{align*}
 %   J(v,w)=\frac{1}{\alpha}v^{1-\alpha}((v+w)^\alpha-v^\alpha)-w.
%\end{align*}
%Then, we have 
%\begin{align*}
 %   \bar{\Psi}_t=(J(V_{t-},\Psi_t)+\Psi_t)V_{t-}^{\alpha-1}.
%\end{align*}
Using the second expression for $\bar{\Psi}$, we have
\begin{equation}\label{zt2}
    \frac{ \bar{\Psi}_t\partial_u\partial_y f}{ \partial_y^2 f}= -\frac{1-{\alpha}{\rho}}{1-{\rho}}\frac{ (J(V_{t},\Psi_t)+\Psi_t)V_{t}^{-{\rho}}}{\frac{1}{1-{\rho}}V_{t}^{1-{\rho}}}Y_t=   -(1-{\alpha}{\rho})\frac{ (J(V_{t},\Psi_t)+\Psi_t)}{V_{t}}Y_t
\end{equation}
Recall  that for the felicity function given by \eqref{barf}, we have 
\begin{align*}
    \epsilon^{{f}}=-\frac{y^{-1}\partial_y {f}}{\partial_y^2 {f}} ={\alpha}, \ \rho^f=\frac{f \partial_u\partial_y f}{\partial_y f}-\partial_u f=\delta.
\end{align*}
Plugging \eqref{zt1} (resp. \eqref{zt2}) to \eqref{consumptionrate}, we finally obtain that
\begin{equation}\label{consumrate}\begin{split}
    c_t=&\frac{1}{\beta} \left[ \left( r - \delta\right) \alpha + \beta\right] Y_t+\frac{\lambda}{\beta}\left({\alpha}(e^\theta-1)-\frac{1-{\alpha}{\rho}}{1-{\rho}}\left(\left(1+\frac{\Psi_t}{V_{t}}\right)^{1-{\rho}}-1\right)\right)Y_t\\
    =&\frac{1}{\beta} \left[ \left( r - \delta\right) \alpha + \beta\right] Y_t+\frac{\lambda}{\beta}\left({\alpha}(e^\theta-1)-(1-{\alpha}{\rho})\frac{ (J(V_{t},\Psi_t)+\Psi_t)}{V_{t}}\right)Y_t.
\end{split}\end{equation}
%\fr{Compare  with   Bank-Riedel. The new term indeed is the last one, and it comes only from stochastic differential utility. }

\begin{remark}
Consider the special case that $\alpha=\frac{1}{\rho}$. Then    $f$ given in \eqref{barf} is the    time-additive aggregator
    \begin{equation}\label{timeadditive}
        f(y,u)=\frac{\delta}{1-\rho}y^{1-\rho}-\delta u.
    \end{equation}
%    where $\delta$ represents the rate of time preference and $\frac{1}{\rho}$ is the elasticity of intertemporal substitution.
Recall that the felicity function $v$ used in \cite{BR} takes the following form
\begin{align*}
    v(t,y)=e^{-\delta t}\frac{1}{\alpha'}y^{\alpha'}.
\end{align*}
If $\alpha'=1-\rho$, the utility generated by $v$ given above and the recursive utility  generated by $f$  given in \eqref{timeadditive} are equivalent. With this felicity function $f$, suppose that \eqref{parameterconstraints} is satisfied, i.e., $$\mu:=\lambda(e^\theta-1)+r+{\beta}\rho-\delta(=\lambda(e^\theta-1)+r+\beta(1-\alpha')-\delta)>0.$$ According to \eqref{consumrate}, we have
\begin{align*}
    c_t=\frac{1}{\beta}\left[\frac{r-\delta}{1-\alpha'}+\beta\right]Y_t+\frac{\lambda (e^\theta-1)}{\beta(1-\alpha')}Y_t=\frac{\lambda(e^\theta-1)+r+\beta(1-\alpha')-\delta}{\beta(1-\alpha')}Y_t.
\end{align*}
For the case that \eqref{parameterconstraints} is not valid, i.e., 
$$\lambda(e^\theta-1)+r+{\beta}{\rho}-\delta\leq 0,$$
by Proposition \ref{lumpconsumptionat0}, the optimal consumption plan is to consume all wealth $w$ at time $0$. We thus recover  the results in Subsection 4.5.2 of \cite{BR}. 
\end{remark}

\section{Conclusion}
In this paper, we consider the intertemporal choice problem with recursive preferences exhibiting local substitution in a jump-diffusion stochastic framework, thus extending the analysis of the deterministic setting in \cite{LRY}. We establish existence of the related stochastic differential utility. We prove existence and uniqueness of the solution as well as sufficient and necessary first-order conditions that allow to characterize optimal intertemporal consumption. The construction of optimal plans is much more involved than in the deterministic case. Taking up ideas from \cite{BR}, we show how one can learn more by using the concept of a minimal level of satisfaction. This minimal level of satisfaction is related to a system of forward-backward equations. The study of these equations might be of independent interest. Our results also show how the deterministic model's insights can be effectively extended to a Poisson price framework. These advancements provide a deeper understanding of optimal consumption and offer a robust foundation for future research in stochastic intertemporal choice models.

\begin{appendix}
    
\end{appendix}

\section{Proofs and Auxiliary Results}
%We first recall   properties for the level of satisfaction. %In the following, for simplicity, we always denote by $L$ a constant depending on $w,T,M$, which may vary from line to line.

 Given our assumptions, $r,\eta,\theta$  are bounded by some constant $M$.  

\begin{lemma}\label{l1} Suppose that Assumption \ref{a1} (iv) holds. For  $w'>0$, there exists a positive constant $L$ depending on $w',T,M$ such that
		\begin{description}
		\item[(i)] For any $C\in\mathcal{X}$ and $t\in[0,T]$, we have% There exists some constant $L\geq 0$, such that
		\begin{displaymath}
		Y_t^C\leq L(1+C_t).
		\end{displaymath}
		\item[(ii)] If $\{C^n\}_{n=1}^\infty\subset \mathcal{X}$ converges almost surely to $C\in\mathcal{X}$ in the weak topology of measures on $[0,T]$, then we have almost surely $Y_t^{C^n}\rightarrow Y_t^C$ for $t=T$ and for every point of continuity $t$ of $C$.
		\item[(iii)]  For  $\P^*\in\mathcal{P}$, we have%Given some $a\in(0,1)$, for any $\P^*\in\mathcal{P}$ and $p>1$ with $a p<1$, we have
		\begin{displaymath}
		\sup_{C\in\mathcal{A}(w')}\E^*[|C_T|]\leq L, \ \sup_{C\in\mathcal{A}(w')}\E[|C_T|^{2\alpha p}]<\infty,
		\end{displaymath}
        where $p$ is the constant in Assumption \ref{a1} (iv).
%		\item[(iv)] For any $t\in[0,T]$, $C\in\mathcal{A}(w)$ and $p>1$ with $\alpha p< 1$, we have
%		\begin{displaymath}
			%\E\left[|U_t^C|^p\right]\leq L.
		%\end{displaymath}
	\end{description}
\end{lemma}

\begin{proof}
	Noting that $\eta$ and $\theta$ are continuous and bounded, it is easy to check that (i) and (ii) hold.     The first conclusion of (iii) follows from the fact that $r$ is bounded and thus $\gamma$ bounded away from zero. By H\"{o}lder's inequality,  we obtain that% for $p>1$ with $\alpha p=1$,
	\begin{displaymath}
		\E[|C_T|^{2\alpha p}]=\E^*\left[\frac{d\P_0}{d\P^*}|C_T|^{2\alpha p}\right]\leq \left(\E^*\left[\left(\frac{d\P_0}{d\P^*}\right)^q\right]\right)^{\frac{1}{q}}\left(\E^*[|C_T|]\right)^{2\alpha p}<\infty,
	\end{displaymath}
	where $2\alpha p+\frac{1}{q}=1$. We finally obtain the desired result.
% We only need to show (iv). By simple calculation, we have
%	\begin{align*}
%		|U_t^C|^p&\leq L\left\{\E_t\left[\bigg|\int_t^T f(s,Y_s^C,0)ds\bigg|^p\right]+\E_t\left[\bigg|\int_t^T (f(s,Y_s^C,0)-f(s,Y_s^C,U_s^C))ds\bigg|^p\right]\right\}\\
%		&\leq L \E_t\left[\int_t^T \left(1+|Y_s^C|^{\alpha p}\right)ds\right]+L \E_t\left[\int_t^T |U_s^C|^p ds\right]\\
%		&\leq L \E_t\left[1+|C_T|^{\alpha p}\right]+L \int_t^T\E_t\left[|U_s^C|^p\right]ds.
%	\end{align*}
%	Taking expectations on both sides yields that
%	\begin{equation}\label{e2}
%		\E[|U_t^C|^p]\leq L \E[1+|C_T|^{\alpha p}]+L\int_t^T \E[|U_s^C|^p]ds.
%	\end{equation}
%	 Applying the Gronwall inequality to \eqref{e2}, we finally get the desired result.
\end{proof}

\begin{lemma}\label{l4}
	Let $C^*$ be optimal for the original problem \eqref{e1} and let $\phi^*=\nabla V(C^*)$.  Then, $C^*$ is also optimal for the following linear problem
	\begin{equation*}\label{6}
		\sup_{C\in\mathcal{A}(w)} \E\left[\int_0^T \phi^*_t dC_t\right].
	\end{equation*}
%	Furthermore, the value of this problem is finite.
\end{lemma}

\begin{proof}
	For any $C\in\mathcal{A}(w)$ and $\varepsilon\in[0,1]$, let $C^\varepsilon=\varepsilon C+(1-\varepsilon)C^*$. For simplicity, set $Y^*=Y^{C^*}$, $Y=Y^C$, $Y^\varepsilon=Y^{C^\varepsilon}$, $U^*=U^{C^*}$, $U=U^C$, $U^\varepsilon=U^{C^\varepsilon}$, $\hat{U}=U^\varepsilon-U^*$, $\hat{Y}=Y^\varepsilon-Y^*$. By the concavity of the felicity function, we have, for any $t\in[0,T)$ and stopping time $\tau$
	\begin{align*}
	&\frac{1}{\varepsilon}\hat{U}_tI_{\{\tau>t\}}\\
    =& \frac{1}{\varepsilon}\E_t\left[I_{\{\tau>t\}}\int_t^\tau (f(s,Y_s^\varepsilon,U_s^\varepsilon)-f(s,Y_s^*,U_s^*))ds+\frac{1}{\varepsilon}\hat{U}_\tau I_{\{\tau>t\}}\right]\\
	\geq &\frac{1}{\varepsilon}\E_t\left[I_{\{\tau>t\}}\int_t^\tau \left(\partial_y f(s,Y_s^\varepsilon,U_s^\varepsilon)\hat{Y}_s+ \partial_u f(s,Y_s^\varepsilon,U_s^\varepsilon)\hat{U}_s\right)ds+\frac{1}{\varepsilon}\hat{U}_\tau I_{\{\tau>t\}}\right]\\
	=&\E_t\left[I_{\{\tau>t\}}\int_t^\tau \left(\partial_y f(s,Y_s^\varepsilon,U_s^\varepsilon)\int_0^s \theta_{s,t}(dC_t-dC_t^*) +\partial_u f(s,Y_s^\varepsilon,U_s^\varepsilon)\frac{1}{\varepsilon}\hat{U}_s\right)ds+\frac{1}{\varepsilon}\hat{U}_\tau I_{\{\tau>t\}}\right],
	\end{align*}
	where the last equality follows from the fact that
	\begin{displaymath}
		Y_s^\varepsilon-Y_s^*=\varepsilon (Y_s-Y_s^*)=\varepsilon \int_0^s \theta_{s,t}(dC_t-dC_t^*).
	\end{displaymath}
	%By Lemma \ref{l2} and Lemma \ref{l3} 
	By Proposition B.1 in \cite{KSS} and noting that $C^*$ is optimal for the utility maximization problem \eqref{e1}, we obtain that
	\begin{displaymath}
		0\geq \frac{1}{\varepsilon}(U^\varepsilon_0-U^*_0)\geq \E\left[\int_0^T \exp\left(\int_0^s \partial_uf(r,Y_r^\varepsilon,U_r^\varepsilon)dr\right)\partial_yf(s,Y_s^\varepsilon,U_s^\varepsilon)\int_0^s \theta_{s,t}(dC_t-dC^*_t)ds\right].
	\end{displaymath}
	Set $\Phi^\varepsilon(t)=\int_t^T \exp(\int_0^s \partial_uf(r,Y_r^\varepsilon,U_r^\varepsilon)dr)\partial_yf(s,Y_s^\varepsilon,U_s^\varepsilon)\theta_{s,t}ds$, $t\in[0,T]$. By the Fubini theorem, the above equation yields that, for any $\varepsilon\in(0,1]$,
	\begin{displaymath}
		\E\left[\int_0^T \Phi^\varepsilon(t)dC_t^*\right]\geq \E\left[\int_0^T \Phi^\varepsilon(t)dC_t\right].
	\end{displaymath}
	By the continuity of the utility function and the level of satisfaction w.r.t consumption, it follows that
	\begin{displaymath}
	   \Phi^*(t):=\Phi^0(t)=\lim_{\varepsilon\rightarrow 0}\Phi^\varepsilon(t), \ t\in[0,T].
	\end{displaymath}
	Since $\int_0^T \Phi^\varepsilon(t)dC_t\geq 0$, applying Fatou's lemma, we have
	\begin{equation}\label{e7}
		\liminf_{\varepsilon\rightarrow 0}\E\left[\int_0^T \Phi^\varepsilon(t)dC_t\right]\geq \E\left[\int_0^T \Phi^*(t)dC_t\right].
	\end{equation}
	We claim that
	\begin{equation}\label{e8}
	\lim_{\varepsilon\rightarrow 0}\E\left[\int_0^T \Phi^\varepsilon(t)dC^*_t\right]= \E\left[\int_0^T \Phi^*(t)dC^*_t\right].
	\end{equation}
	In order to get this result, it is sufficient to show that the family
	\begin{displaymath}
		I^\varepsilon=\int_0^T \Phi^\varepsilon(t)dC^*_t,  \ \varepsilon\in[0,\frac{1}{2}]
	\end{displaymath}
	has a $\P_0$-integrable upper bound. Recalling that the utility $U_t^C$ is concave and monotone with respect to consumption $C$ for any $t\in[0,T]$ (see Theorem \ref{propertyofutility}), we have
	\begin{displaymath}
		\varepsilon U_t+(1-\varepsilon)U_t^*\leq U_t^\varepsilon\leq U_t^{\widetilde{C}},
	\end{displaymath}
	where $\widetilde{C}=C+C^*\in\mathcal{A}(2\omega)$. Note that $Y^\varepsilon_t\geq \frac{1}{2}Y^*_t$ and  for any $\varepsilon\in[0,\frac{1}{2}]$ and $t\in[0,T]$. By the assumption on $f$, it is easy to check that
	\begin{align*}
		I^\varepsilon=& \int_0^T \exp\left(\int_0^s \partial_uf(r,Y_r^\varepsilon,U_r^\varepsilon)dr\right)\partial_yf(s,Y_s^\varepsilon,U_s^\varepsilon)\int_0^s\theta_{s,t}dC_t^* ds\\
		\leq &L \int_0^T \partial_yf(s,Y_s^\varepsilon,U_s^\varepsilon)Y_s^* ds\leq 2L\int_0^T \partial_yf\left(s,\frac{1}{2}Y_s^*,U_s^\varepsilon\right)\frac{1}{2}Y_s^*ds\\
		\leq &L \int_0^T \left[f(s,\frac{1}{2}Y_s^*, U_s^\varepsilon)-f(s,0,U_s^\varepsilon)\right]ds\\
		\leq &L \int_0^T \left[1+|Y_s^*|^\alpha+|U_s^*|+|U_s|+|U_s^{\widetilde{C}}|\right]ds.
	\end{align*}
	By Lemma \ref{l1} and \eqref{claim1}, we get the required upper bound for $I^\varepsilon$. Hence, the claim \eqref{e8} hold true. Eqs. \eqref{e7} and \eqref{e8} imply that
	\begin{displaymath}
		\E\left[\int_0^T \Phi^*(t)dC^*_t\right]\geq \E\left[\int_0^T \Phi^*(t)dC_t\right].
	\end{displaymath}
	By Theorem (1.33) in Jacod \cite{J79}, we may replace $\Phi^*$ by its optional projection which coincides with $\nabla V(C^*)=\phi^*$. The proof is complete.
\end{proof}

\begin{remark}\label{r3}
	In fact, for any stopping time $\tau$, the following conditional flat off condition still holds for the optimal consumption plan $C^*$:
	\begin{displaymath}
		\E_\tau\left[\int_\tau^T (\nabla V(C^*)(t)-M\psi_t)dC_t^*\right]=0.
	\end{displaymath}
	Otherwise, there exists a stopping time $\tau$ such that $P(A)>0$, where $$A=\left\{\E_\tau\left[\int_\tau^T \nabla V(C^*)(t)dC_t^*\right]<M\E_\tau\left[\int_\tau^T \psi_t dC_t^*\right]\right\}\in\mathcal{F}_\tau.$$ 
 Since for any $t\in[0,T]$, $\nabla V(C^*)(t)\leq M\psi_t$, it is easy to check that
	\begin{displaymath}
		0=\E\left[\int_0^T (\nabla V(C^*)(t)-M\psi_t)dC^*_t\right]\leq \E\left[\int_\tau^T (\nabla V(C^*)(t)-M\psi_t)dC^*_t\right]\leq 0,
	\end{displaymath}
	which implies that $\E[\int_\tau^T (\nabla V(C^*)(t)-M\psi_t)dC^*_t]=0$. By simple calculation, we obtain that
	\begin{align*}
		&\E\left[\int_\tau^T (\nabla V(C^*)(t)-M\psi_t)dC^*_t\right]\\
		=&\E\left[\E_\tau\left[\int_\tau^T (\nabla V(C^*)(t)-M\psi_t)dC^*_t\right]I_A+\E_\tau\left[\int_\tau^T (\nabla V(C^*)(t)-M\psi_t)dC^*_t\right]I_{A^c}\right]<0,
	\end{align*}
	which is a contradiction.
\end{remark}


\begin{thebibliography}{99}
	\bibitem{BE} Bank, P. and El Karoui, N. (2004) A stochastic representation theorem with applications to optimization and obstacle problems. Ann.  Probab., 32, 1030-1067.
	
	\bibitem{BR00}  Bank, P. and Riedel, F. (2000) Non-time additive utility optimization-the case of certainty. J. Math. Econom., 33, 271-290.
	
	\bibitem{BR} Bank, P. and Riedel, F. (2001) Optimal consumption choice with intertemporal substitution. Ann. Appl. Probab., 11(3), 750-788.

    \bibitem{BBP} Barles, G., Buckdahn, R. and Pardoux, E. (1997) Backward stochastic differential equations and integral-partial differential equations. Stoch. Stoch. Rep., 60(1-2), 57-83.
	
%	\bibitem{CE} Chen, Z. and Epstein, L. (2002) Ambiguity, risk and asset returns in continuous time. Econometrica, 70, 1403-1443.
    \bibitem{CT} Cont, R. and Tankov, P. (2004) Financial Modelling with Jump Processes. Chapman and Hall/CRC Press, London.

   \bibitem{Delong} Delong, L. (2013) Backward Stochastic Differential Equations with Jumps and Their Actuarial and Financial Applications. Springer-Verlag, London.
	
	\bibitem{DE} Duffie, D. and Epstein, L. (1992) Stochastic differential utility. Econometrica, 60, 353-394.
	
	\bibitem{EPQ} El Karoui, N., Peng, S. and Quenez, M. (1997) Backward stochastic differential equations in finance. Math. Financ., 7(1), 1-71.

   \bibitem{E87} Epstein, L. (1987) The global stability of efficient intertemporal allocations. Econometrica, 55, 39-358.
	
	\bibitem{HH}  Hindy, A. and Huang, C.-F. (1993) Optimal consumption and portfolio rules with durability and local substitution. Econometrica, 61, 85-121.
	
	\bibitem{HHK} Hindy, A., Huang, C.-F. and Kreps, D. (1992) On intertemporal preference in continuous time: the case of certainty. J. Math. Econom., 21, 401-440.
	
	\bibitem{J79} Jacod, J. (1979) Calcul Stochastique et Probl\`{e}mes de Martingales. Lecture Notes in Math. 714, Springer, Berlin.
	
	\bibitem{K99} Kabanov, Y. (1999) Hedging and liquidation under transaction costs in currency markets. Finance and Stochastics, 2, 237-248.
	
	\bibitem{K67} Koml\'{o}s, J. (1967) A generalization of a problem of Steinhaus. Acta Math. Acad. Sci. Hung., 18, 217-229.


	
    \bibitem{KSS} Kraft, H., Seiferling, T. and Seifried, F.T. (2017) Optimal consumption and investment with Epstein-Zin recursive utility. Finance and Stochastics, 21, 187-226.

     \bibitem{KS} Kraft, H. and Seifried, F.T. (2010) Foundations of continuous-time recursive utility: differentiability and normalization of certainty equivalents. Math. Finan. Econ., 3, 115-138.

    \bibitem{LRY} Li, H., Riedel, F. and Yang, S. (2024) Optimal consumption for recursive preferences with local substitution--the case of certainty. Journal of Mathematical Economics, 110, 102932.

    \bibitem{Ma} Ma, C. (2000) An existence theorem of intertemporal recursive utility in the presence of L\'{e}vy jumps. Journal of Mathematical Economics, 34, 509-526.

	\bibitem{PP90} Pardoux, E. and Peng, S. (1990) Adapted solution of a backward stochastic differential equation. Systems Control Lett. 14, 55-61.
	
%	\bibitem{P97} Peng, S. (1997) BSDE and related g-expectations, in Pitman Research Notes in Mathematics Series, No. 364, ``Backward Stochastic Differential Equations", Ed. by El Karoui $\&$ L. Mazliak, 141-159.
	
	\bibitem{RS} Riedel, F. and Su, X. (2011) On irreversible investment. Finance and Stochastics, 15, 607-633.

 \bibitem{Royer} Royer, M. (2006) Backward stochastic differential equations with jumps and related non-linear expectations. Stoch. Proc. Appl., 116, 1358-1376.

 \bibitem{S} Skiadas, C. (2013) Smooth ambiguity aversion toward small risks and continuous-time recursive utility. Journal of Political Economy, 121(4), 775-792.

 \bibitem{TL} Tang, S. and Li, X. (1994) Necessary conditions for optimal control of stochastic systems with random jumps. SIAM J. Control Optim., 32(5), 1447–1475.
\end{thebibliography}
\end{document}